\documentclass[11pt, a4paper]{amsart}

\usepackage{amsfonts,amsmath,amssymb,amscd}
\usepackage{txfonts}
\usepackage{bbm}

\usepackage{graphicx}
\usepackage{stmaryrd}
\usepackage[all]{xy}

\newtheorem{theorem}{Theorem}[section]
\newtheorem{lemma}[theorem]{Lemma}
\newtheorem{prop}[theorem]{Proposition}
\newtheorem{corollary}[theorem]{Corollary}

\newtheorem*{claim}{Claim}

\theoremstyle{definition}

\newtheorem{rem}[theorem]{Remark}

\usepackage{mathrsfs}
\usepackage{color}
\newcommand{\tc}[1]{\textcolor{black}{#1}}

\newcommand\pf{\begin{proof}}
\newcommand\epf{\end{proof}}

\let\oldtocsection=\tocsection

\let\oldtocsubsection=\tocsubsection

\renewcommand{\tocsection}[2]{\hspace{0em}\oldtocsection{#1}{#2}}

\renewcommand{\tocsubsection}[2]{\hspace{1em}\oldtocsubsection{#1}{#2}}

\numberwithin{equation}{section}

\title[Smoothness of commutative Hopf algebras]
{Smoothness of commutative Hopf algebras}

\author[K.~Egami]{Kensuke Egami}
\address{Kensuke Egami,
Graduate School of Pure and Applied Sciences, 
University of Tsukuba, Ibaraki 305-8571, Japan}
\email{gmknsk@gmail.com}

\author[A.~Masuoka]{Akira Masuoka}
\address{Akira Masuoka,
Department of Mathematics, 
University of Tsukuba, 
Ibaraki 305-8571, Japan}
\email{akira@math.tsukuba.ac.jp}

\author[K.~Suzuki]{Kenta Suzuki}
\address{Kenta Suzuki,
NIPPI Corporation,
3175 Showa-machi, Kanazawa-ku, Yokohama 236-8540, JAPAN}
\email{suzukikentaci409c@docomo.ne.jp}

\begin{document}

\begin{abstract}
Hopf algebras, most generally in a semisimple abelian symmetric monoidal category, 
are here supposed to be commutative but not to be of finite-type, and their (equivariant) smoothness are discussed. 
Given a Hopf algebra $H$ in a category such as above, it is proved that the following are equivalent: 
(i)~$H$ is smooth as an algebra; (ii)~$H$ is smooth as an $H$-comodule algebra; (iii)~the product morphism
$S_H^2(H^+) \to H^+$ defined on the 2nd symmetric power is monic. 
Working over a field $k$ of characteristic zero, we prove: (1)~every ordinary Hopf algebra, i.e., such
in the category $\mathsf{Vec}$ of vector spaces, satisfies the equivalent conditions (i)--(iii) and some others;
(2)~every Hopf algebra in the category $\mathsf{sVec}$ of super-vector spaces has a certain property that
is stronger than (i). In the case where $\operatorname{char}k=p>0$, there
are shown weaker properties of ordinary Hopf algebras and of Hopf algebras in $\mathsf{sVec}$ or in the 
ind-completion $\mathsf{Ver}_p^{\mathrm{ind}}$ of the Verlinde category.
\end{abstract}

\maketitle




\section{Introduction: background and main theorems}\label{Sec:intro}

This paper discusses smoothness of commutative Hopf algebras not necessarily of finite type,
which include commutative Hopf-algebra objects in general categories with appropriate properties.
We thus say ``of finite type" for ``finitely generated".
We have three main theorems which are explicitly formulated in the following subsections.

\subsection{Smoothness of ordinary Hopf algebras}\label{S1.1}
Let us work over an arbitrary field $k$, and vector spaces, linear maps and tensor products mean those over $k$.
Algebras and Hopf algebras, which indicate those over $k$, 
are supposed to be commutative unless otherwise stated. 

An algebra is said to be \emph{smooth} (over $k$) \cite[Section 9.3.1]{Wei}, if given an algebra $R$ and a nilpotent ideal $I$
of $R$, every algebra map $f : A \to R/I$ can \emph{lift} to some algebra map $g : A \to R$, i.e., the $g$, 
composed with the natural projection $R \to R/I$, turns into $f$. 
It is known that in the case where $A$ is Noetherian, 
the condition is equivalent 
to saying that $A$ is geometrically regular, i.e., it remains regular after base extension to the algebraic closure of $k$. 
The definition above is suited when we discuss, as we will do, smoothness of algebras not necessarily 
of finite type.
As another advantage the definition, which concerns algebras in the symmetric monoidal category $\mathsf{Vec}$
of vector spaces, can directly extend to those in general categories with appropriate properties. 
For example, given a Hopf algebra $H$, we can (and we do) discuss smoothness of \emph{$H$-comodule algebras},
i.e., algebras in the symmetric monoidal category
of (right or left) $H$-comodules. 
We say that an $H$-comodule algebra $A$ is \emph{$H$-smooth}
(or \emph{equivariantly smooth}, without referring to the Hopf algebra) if given an $H$-comodule algebra and a nilpotent $H$-costable ideal $I$ of $R$, 
every $H$-colinear algebra map $A \to R/I$ can lift to such a map $A \to R$. 
Here and in what follows, comodules are meant to be 
right comodules unless otherwise stated. 
Our interest is in the case where $A$ is $H$, itself. 

Let $H$ be a Hopf algebra. Let $H^+=\operatorname{Ker}(\varepsilon_H:H\to k)$ denote the augmentation
ideal, and let $S^2_H(H^+)$ denote the 2nd symmetric power of the $H$-module $H^+$.
Let us consider the conditions:
\begin{itemize}
\item[(a)] 
$H$ is smooth,
\item[(b)] 
$H$ is $H$-smooth,
\item[(c)] 
$H$ is $Q$-smooth for every quotient Hopf algebra $Q$ of $H$ such that
$H$ is coflat as a $Q$-comodule, i.e., the co-tensor product
$H\square_Q$ (see \eqref{Ecotens}) is exact as a functor,
\item[(d)] 
the $H$-linear map
\[
\mu_H : S^2_H(H^+)\to H^+
\]
which is induced by the product $H^+\otimes_H H^+\to H^+$, $a \otimes_H b\mapsto ab$ is injective,
\item[(e)] 
$H$ is geometrically reduced, i.e., it remains reduced after the base extension to the
algebraic closure of $k$,
\end{itemize}
and in addition, in case $\operatorname{char}k=p>0$,
\begin{itemize}
\item[(f)]
the Frobenius map
\[
F_H : k^{1/p}\otimes H \to H,\quad F_H(a^{1/p}\otimes h)=ah^p
\]
is injective, where $k^{1/p}\otimes$ presents the base extension 
along the isomorphism $k \to k^{1/p}$, $a\mapsto a^{1/p}$ of fields.
\end{itemize}

\begin{rem}
Recall that the Hopf algebras form a category anti-isomorphic to the category of affine group schemes.
Let $\mathfrak{G}=\operatorname{Spec}H$ be the affine group scheme corresponding to $H$.
\begin{itemize}
\item[(1)] Condition (a) is restated so that
\begin{itemize}
\item[(a)] Given an algebra $R$ and a nilpotent ideal $I$ of $R$, the group map
\begin{equation}\label{EGR}
\mathfrak{G}(R)\to \mathfrak{G}(R/I)
\end{equation}
associated with the projection $R \to R/I$ is surjective.
\end{itemize}
\item[(2)] A right (resp., left) $H$-comodule algebra is identified with a left (resp., right) $\mathfrak{G}$-equivariant
algebra. A right $H$-comodule (or left $\mathfrak{G}$-equivariant) algebra  turns into a left (resp., right) one by twisting the coaction through the antipode of $H$ (resp., the action through the inverse of $\mathfrak{G}$), 
and vise versa. Under the resulting one-to-one correspondence, $H$ corresponds to $H$, itself. Therefore, 
as for (b), $H$ is $H$-smooth as a right $H$-comodule algebra if and only if it is so as a left $H$-comodule algebra. 
\item[(3)]
By a similar reason, as for (c), $H$ is coflat as a right $Q$-comodule if and only if it is so as a left $Q$-comodule.
Recall that every closed subgroup scheme $\mathfrak{H}$ of $\mathfrak{G}$ is uniquely in the form
$\mathfrak{H}=\operatorname{Spec}Q$, where $Q$ is a quotient Hopf algebra of $H$. 
It is known (see \cite[p.452]{T0}, \cite[Theorem 2 and the following Remark]{Doi}) that 
the equivalent coflatness conditions
above are equivalent to saying that the quotient sheaf $\mathfrak{G}/\mathfrak{H}$ in the fpqc topology
is an affine scheme; see also Remark \ref{Rcoflatinj} below.
\end{itemize}
\end{rem}

The first main theorem is the following. 

\begin{theorem}\label{T1st}
We have the following:
\begin{itemize}
\item[(1)]
If $\operatorname{char}k=0$, every Hopf algebra $H$ satisfies Conditions (a)--(e) above. 
\item[(2)]
If $\operatorname{char}k> 0$, the Conditions (a)--(f) are equivalent to each other. 
\end{itemize}
\end{theorem}

Suppose $\operatorname{char}k=0$. 
It is well known that every Hopf algebra satisfies (e),  and in addition, (a) at least when 
it is of finite type; see Section \ref{subsec:charzero}. 

Suppose $\operatorname{char}k=p>0.$
Then there exist many Hopf algebras that do not satisfy the conditions; in particular, there are known simple examples of Hopf algebras which are reduced, but not geometrically reduced, see \cite[Remark 39.8.5]{Stack}, for example. 
It is proved by Takeuchi \cite[Proposition 1.9]{T} that (f)$\Rightarrow$(a)$\Rightarrow$(d), and these conditions are equivalent to each other 
in case $H$ is of finite type.
Likely, it has been long unknown whether they are equivalent in general. It is also proved by \cite[Corollary 1.6]{T}
that every Hopf algebra that satisfies (f) has a certain stronger smoothness property; the result will be proved
by Theorem \ref{T3rd} (2) below in a generalized situation, see Section \ref{S1.4}. 

Theorem \ref{T1st} will be proved in Section \ref{Sec3}. It is after we prove the equivalence 
of Conditions (a)--(d) for Hopf algebras in a general category with appropriate properties;
more details will be found in the following subsection. 

\subsection{Smoothness of Hopf algebras in categories}\label{S1.2}
Let $\mathscr{C}$ be a symmetric monoidal category satisfying the assumptions:
\begin{itemize}
\item[(I)] $\mathscr{C}$ is, as a category, semisimple abelian, i.e., abelian and 
every object is a coproduct of simple objects, and
\item[(II)] the tensor-product functor is bi-additive, which is then necessarily exact by (I).
\end{itemize}

The category $\mathsf{Vec}$ of vector spaces is an obvious example of $\mathscr{C}$. 
Conditions (a)--(d) concerning Hopf algebras in $\mathsf{Vec}$
make sense more generally for those Hopf algebras in $\mathscr{C}$ which are commutative by assumption; 
this will be verified in detail in the paragraph following Proposition \ref{PH2C}. 
The second main theorem 
is the following. 

\begin{theorem}\label{T2nd}
For a Hopf algebra $H$ in $\mathscr{C}$, Conditions (a)--(d) are equivalent to each other.
\end{theorem}

This will be proved in the following Section \ref{Sec2}. 
One will see that among the conditions, a key is (b), the equivariant smoothness; the background 
of the notion will be seen in the following subsection. Important is the fact that
Condition (b) is equivalent to saying that for every object $X$ of $\mathscr{C}$, the symmetric 2nd Hochschild cohomology (constructed in $\mathscr{C}$) of $H$ with coefficients in $X$ vanishes,
\[
H_s^2(H,X)_{\mathscr{C}}=0,
\]
where $X$ is regarded as a trivial $H$-module through the counit of $H$. 

As was mentioned above, we will prove the first main theorem, Theorem \ref{T1st},
after proving (and by using) the second one, Theorem \ref{T2nd}. A point to prove
Part 2 (the positive characteristic case) of Theorem \ref{T1st} is to show 
that the above cohomology-vanishing condition 
for finite-type (ordinary) Hopf algebras implies the same condition for Hopf algebras in general. 
This will be shown by proving that the restriction map for symmetric 2nd Hochschild cohomologies
(constructed in $\mathsf{Vec}$)
\begin{equation}\label{Eintres}
\mathrm{res} : H_s^2(H,k) \to H_s^2(J,k)
\end{equation}
is surjective, where $H$ is an ordinary Hopf algebra, and $J$ is a Hopf subalgebra of $H$; see Proposition \ref{Pres}. 
The proof depends on the fact that the 2nd cohomology in question is naturally isomorphic to the 1st Andr\'{e}-Quillen
cohomology, and uses some results on the latter cohomology that include substantial ones by Avramov \cite{A}. 


\subsection{Background around the equivariant smoothness}\label{S1.3}
The notion of \emph{$H$-crossed products}, where $H$ is a not necessary commutative ordinary Hopf algebra,
was introduced by Blattner et al. \cite{BCM}; it generalizes group crossed products, for which
$H$ is a group algebra as a special case. 
To be more explicit, an $H$-crossed product, $R \rtimes_{\sigma}\! H$,  over a (not necessarily commutative) algebra $R$ is constructed on the $H$-comodule $R\otimes H$, subject to an action (in a weak sense) by $H$ on $R$ and 
a Hopf 2-cocycle $\sigma : H \otimes H \to R$. 
As was shown by Doi et al. \cite{DT}, the $H$-crossed products are intrinsically characterized as \emph{$H$-cleft extensions}, i.e., $H$-comodule algebras with 
invertible (with respect to the convolution-product) $H$-colinear map from $H$.
The second-named author \cite{M1} pointed out (in fact, in the dual situation) that if the $R$ has an augmentation $R \to k$ with square-zero kernel, the \emph{augmented} $H$-cleft extensions over $R$ are classified by the 2nd Hochschild cohomologies. 
As a result it was essentially proved by \cite[Theorem 4.1]{M1} (see also \cite[Theorem 1.2, Proposition 1.10]{MO}) that if $H$ is smooth, then it is $H$-smooth in the present words; the converse holds, as well, but this is
rather easy, as will 
be seen in Proposition \ref{PMOTeq} (modified into the \tc{commutative} situation);
see also \cite[Proposition 3.19]{MOT}. 
The last mentioned result of \cite{M1} was re-proved by Ardizonni et al. \cite[Theorem 5.9 a)]{AMS} based on  categorical observation using the notion of (co)separability. 
We reach cohomologies directly (not through augmented cleft extensions), and deduce a key implication, (a)$\Rightarrow$(d), of Theorem \ref{T2nd} from an isomorphism of the relevant cohomologies; see Proposition \ref{PH2}.

Nevertheless, it is a useful method of computing 2nd Hochschild cohomologies of a Hopf algebra $H$ to classify the corresponding augmented $H$-cleft extensions, as was shown in \cite{M3} in the case
where $H$ is the quantized universal enveloping algebra. In the present situation in which 
the relevant (Hopf) algebras are supposed to be commutative, we
discuss, as was mentioned above, the symmetric 2nd Hochschild cohomology $H_s^2(H, k)$ with coefficients in the trivial $H$-module $k$, which is precisely in one-to-one correspondence with the set of the isomorphism classes of 
all  (commutative) augmented $H$-cleft extension over the algebra $k[T]/(T^2)$ of dual numbers. This is the topic of Section \ref{Sec4}. Section \ref{subsec:4.1} reviews the above-mentioned one-to-one correspondence.
It is applied, in the following two subsections, to compute $H_s^2(H, k)$ of some examples of Hopf algebras $H$
in positive characteristic.
In Section \ref{subsec:sample1}, $H$ is supposed to be the Hopf algebra given in \cite[Example 1.7]{T}
which is generated by an $\infty$-sequence of primitive elements. In Section \ref{subsec:sample2}, $H$
is supposed to be a finite-type group algebra, and the surjectivity of 
the restriction map \eqref{Eintres} for such an $H$ is verified by explicit computation.

\subsection{Stronger smoothness properties}\label{S1.4}
Beside the obvious $\mathsf{Vec}$, the category $\mathscr{C}$ discussed in Section \ref{S1.2}
has as examples, 
\begin{equation}\label{Ecat}
\mathsf{sVec}\quad \text{and}\quad \mathsf{Ver}_p^{\mathrm{ind}},
\end{equation}
which denote the category of super-vector spaces and the ind-completion of the Verlinde category
$\mathsf{Ver}_p$ in positive characteristic $p$ (see \cite{O}, \cite{Ven}), respectively. 
These categories and the associated tensor-product functors are $k$-linear, where $k$ denotes the base field. 
Both naturally include $\mathsf{Vec}$ as a symmetric monoidal full subcategory. Recently, they are attracting attention due to the following remarkable results which assume that $k$ is algebraically closed:
in case $\operatorname{char}k=0$, Deligne \cite{D} proved that every $k$-linear artinian rigid symmetric monoidal category $\mathscr{T}$, such that (i)~the tensor-product functor is $k$-bilinear, 
(ii)~the unit object is simple, and (iii)~$\mathscr{T}$ has moderate growth, has 
a fiber
functor 
into the category $\mathsf{svec}$ of finite-dimensional super-vector spaces, whence it is equivalent to
the category of comodules in $\mathsf{svec}$ over some Hopf algebra in $\mathsf{sVec}$;
in case $\operatorname{char}k=p>0$, 
Coulembier et al. \cite{CEO} proved that the analogous conclusion in which $\mathsf{svec}$ is replaced by
$\mathsf{Ver}_p$ holds, if (and only if) (iv)~$\mathscr{T}$
is Frobenius and exact, in addition. Here are remarks: 
$\mathsf{sVec}$ is the ind-completion $\mathsf{svec}^{\mathrm{ind}}$ of $\mathsf{svec}$;
Condition (iv) is necessarily satisfied if $\mathscr{T}$ is semisimple.

We will show that Hopf algebras, which are commutative by assumption, in these categories have
some properties stronger than smoothness. 
The familiar words ``(Hopf) superalgebras" meaning (Hopf) algebras in $\mathsf{sVec}$ will not be used in this paper. 
Algebras of $\mathsf{sVec}$
in $\operatorname{char}k=2$ are supposed to satisfy the assumption that every even element is central and every odd element is square-zero, which is stronger than the (super-)commutativity. 
For $\mathsf{Ver}_p^{\mathrm{ind}}$, we assume that $k$ is an algebraically closed
field of characteristic $p >0$. Since $\mathsf{Ver}_2^{\mathrm{ind}}=\mathsf{Vec}$
and $\mathsf{Ver}_3^{\mathrm{ind}}=\mathsf{sVec}$, as is shown in \cite[Example 3.2]{O}, we assume $p \ge 5$, in addition. Then, $\mathsf{Ver}_p^{\mathrm{ind}}$ includes properly $\mathsf{sVec}$ as a symmetric monoidal 
full subcategory. 

The third main theorem is the following. 

\begin{theorem}\label{T3rd}
We have the following.
\begin{itemize}
\item[(1)]
Assume $\operatorname{char}k=0$. 
Let $H$ be a Hopf algebra in $\mathsf{sVec}$. Given an algebra
$R$ in $\mathsf{sVec}$ and its nil ideal $I$, i.e., an ideal consisting of nilpotent elements, 
every algebra morphism $H \to R/I$
can lift to some algebra morphism $H \to R$.
\item[(2)]
Assume $\operatorname{char}k=p>0$, and let us work in $\mathsf{sVec}$
or $\mathsf{Ver}_p^{\mathrm{ind}}$.
Let $H$ be a Hopf algebra in the category, such that the naturally associated, ordinary quotient Hopf algebra
$\overline{H}$ (see \eqref{EHbar}, \eqref{EHbar1}) is smooth. 
Given an algebra
$R$ and a bounded nil ideal $I$ of $R$,  
every algebra morphism $H \to R/I$
can lift to some algebra morphism $H \to R$.
\end{itemize}
\end{theorem}

The conclusions above are restated so that the group map
$\mathfrak{G}(R)\to \mathfrak{G}(R/I)$ analogous to \eqref{EGR} is surjective,
where $\mathfrak{G}=\operatorname{Spec} H$ denotes the affine group scheme corresponding to $H$, i.e., 
the group-valued functor
which is defined on the category of algebras in the respective category, and is represented by $H$. 

As for Part 2, an ideal $I$ is said to be \emph{bounded nil}, if there is an integer $n >0$
satisfying the condition: $X^n=0$ for every simple sub-object $X$ of $I$. 
In $\mathsf{sVec}$ in arbitrary characteristic, an ideal $I$ of an algebra $R$ is (bounded) nil if and only
if in the even component $R_0$ of $R$, which is an ordinary algebra, the ideal $I_0$ of $R_0$ is (bounded) nil. 
In $\mathsf{Ver}_p^{\mathrm{ind}}$, the situation is the same; see the third paragraph of Section \ref{subsec:charp}.

These situations, combined with the tensor product decomposition of a Hopf algebra in the category (see \eqref{Etenpr}, \eqref{Etenpr1}), 
allow us to reduce the proof of 
Theorem \ref{T3rd} to the case where the category is $\mathsf{Vec}$. The proof will be given in Section \ref{Sec5}. 
Crucial for the proof of Part 2 is Takeuchi's result \cite[Corollary 1.6]{T} which was referred to in Section \ref{S1.1}
and is now explicitly stated: an ordinary Hopf
algebra in positive characteristic which satisfies Condition (f) has the stronger smoothness property which is an obvious analogue in $\mathsf{Vec}$ of the one stated in Part 2.
Example 1.7 of \cite{T}, which was also referred to before, shows that Part 2 cannot be strengthen 
with ``bounded nil" relaxed to ``nil" as in Part 1; see the first paragraph of Section \ref{subsec:sample1}.

The same idea as proving Part 1 will be used in the Appendix, to give a simple proof
of the known result that in $\mathsf{Vec}$, 
every torsor under an affine group scheme over an algebraically closed field is trivial. 
Recently, Wibmer \cite{Wib} gave an elementary proof to it, but ours would be simpler. 
We remark that the result holds more generally in $\mathsf{sVec}$ and in $\mathsf{Ver}_p^{\mathrm{ind}}$;
see Remark \ref{RTan}. 

\begin{rem}
We do not know whether the equivariant analogue of Theorem \ref{T3rd}, which replaces
``algebras" with ``$H$-comodule algebras", is true. One sees from the proof of Proposition
\ref{PMOTeq} that
the analogue is a stronger statement, i.e., it implies the theorem.
\end{rem} 

\section{Proof of Theorem \ref{T2nd}}\label{Sec2}

To prove Theorem \ref{T2nd}, we 
let $\mathscr{C}$ be a symmetric monoidal category such that (I)~$\mathscr{C}$ is semisimple abelian,
and (II)~the tensor-product functor is bi-additive, as in Section \ref{S1.2}. 
Concerning (I), what is actually needed in this section is 
the property of $\mathscr{C}$ that it is \emph{split}, i.e., every short exact sequence
is split; indeed, this property is weaker than the semisimplicity. 
The tensor product, the unit object and the symmetry of 
$\mathscr{C}$ will be denoted by $\otimes$, $\mathbf{1}$ and 
\[
c_{X,Y} : X \otimes Y \overset{\simeq}{\longrightarrow} Y \otimes X,\quad X, Y\in \mathscr{C},
\]
respectively. 
The unit or associativity constraints, when they are recognized from the context, are presented as 
identity morphisms. 

In what follows, all (Hopf) algebras in $\mathscr{C}$ are supposed to be commutative, unless otherwise stated. 
Let $A$ be an algebra in $\mathscr{C}$; the product and the unit will be denoted by
$m_A : A\otimes A \to A$ and $u_A : \mathbf{1}\to A$, respectively. 
By an \emph{$A$-module} we mean a left $A$-module $M$
in $\mathscr{C}$. It is naturally identified with a right $A$-module so that 
\[
\begin{xy}
(0,0)   *++{A\otimes M}  ="1",
(28,0)  *++{M\otimes A} ="2",
(14,-14) *++{M} ="3",
{"1" \SelectTips{cm}{} \ar @{->}^{c_{A,M}} "2"},
{"1" \SelectTips{cm}{} \ar @{->}_{a_M} "3"},
{"2" \SelectTips{cm}{} \ar @{->}^{a'_M} "3"}
\end{xy}
\]
commutes; here and in what follows, $a_M$ (resp., $a'_M$) denotes the left (resp., right) $A$-action on $M$. 
In addition, $M$ is regarded as an $(A,A)$-bimodule with the special property that the left and the right $A$-actions
are identified as above. 

A \emph{square-zero extension} of $A$ is an algebra $E$ in $\mathscr{C}$, equipped with an 
algebra morphism $p : E \to A$, which is epic as a morphism of $\mathscr{C}$,
and is such that the kernel $\operatorname{Ker} p$ is square-zero, $(\operatorname{Ker} p)^2=0$.
The last condition ensures that $\operatorname{Ker} p$ naturally turns into an $A$-module 
through $p$. Two square-zero extensions, $(E, p)$ and $(E',p')$, of $A$ are said to be \emph{isomorphic},
if there is an algebra isomorphism $E\to E'$ compatible with $p$ and $p'$. 

Given an $A$-module $M$, a \emph{Hochschild extension} of $A$ by $M$ is a short
exact sequence 
\begin{equation}\label{Eext}
\langle E\rangle : 0 \to M \overset{i}{\longrightarrow} E \overset{p}{\longrightarrow} A \to 0
\end{equation}
in $\mathscr{C}$ such that $(E, p)$ is a square-zero extension of $A$, and the isomorphism 
$i : M \to \operatorname{Ker} p$
is $A$-linear. An \emph{equivalence} $\langle E\rangle \to \langle E'\rangle$ between two Hochschild extensions of $A$ by 
$M$ is an algebra morphism (necessarily, isomorphism) $E \to E'$ between the middle terms which induces the identity morphisms
of $A$ and of $M$. 

By using the splitting property of $\mathscr{C}$
we see that the set 
\[
\operatorname{Ext}(A,\ M)_{\mathscr{C}}
\]
of all equivalence classes of the Hochschild 
extensions of $A$ by $M$ is in one-to-one correspondence with 
the symmetric 2nd Hochschild cohomology group
\begin{equation}\label{EH2AM}
H^2_s(A,\ M)_{\mathscr{C}},
\end{equation}
which is constructed in an obvious manner, 
generalizing the familiar one in the case where $\mathscr{C}=\mathsf{Vec}$. 
Moreover, $\langle E\rangle \mapsto (E, p)$ gives rise to a bijection from $\operatorname{Ext}(A,M)_{\mathscr{C}}$ to the set of
all isomorphism classes of the square-zero extensions of $A$ with kernel isomorphic to $M$. 
The zero cohomology class in 
$H^2_s(H,\ M)_{\mathscr{C}}$ corresponds to the equivalence class of the Hochschild
extensions $\langle E\rangle$, and also to the isomorphism class of the square-zero extensions $(E,p)$, such that
the cokernel $p : E \to A$ splits as an algebra morphism of $\mathscr{C}$. 

Let $\langle E\rangle$ be a Hochschild extension such as in \eqref{Eext}.
Choosing a section $j : A \to E$ of $p$ in $\mathscr{C}$,
we have $M\oplus\operatorname{Im}j=E$, whence $E$ is thus identified, as an object of $\mathscr{C}$, with 
$M\oplus A$ through $j$, so that 
\begin{equation}\label{EEMA}
E=M\oplus A.
\end{equation}
The one-to-one correspondence above is such as follows: the unit $u_E : \mathbf{1} \to E$
and the product $m_E : E\otimes E\to E$ are given uniquely
by a symmetric Hochschild 2-cocycle $s : A \otimes A \to M$ so that
\begin{equation}\label{Eum}
u_E=(-s\circ(u_A\otimes u_A),\ u_A),\quad
m_E=\begin{cases} (s,~m_A) &\text{on}\ A\otimes A;\\
(a_M,~0) &\text{on}\ A\otimes M;\\
(a'_M,~0) &\text{on}\ M\otimes A;\\
\ 0 &\text{on}\ M\otimes M.
\end{cases}
\end{equation}

Let $H$ be a Hopf algebra in $\mathscr{C}$. The coproduct and the counit will be denoted by
$\Delta_H : H \to H\otimes H$ and $\varepsilon_H : H \to \mathbf{1}$, respectively. 
We have the category $\mathscr{C}^H$ of right $H$-comodules
in $\mathscr{C}$, which naturally turns into an abelian symmetric monoidal category with respect to the
tensor product $\otimes$ and the unit object $\mathbf{1}$, just as in the familiar 
case where $\mathscr{C}=\mathsf{Vec}$ or $\mathsf{sVec}$. 
Due to the splitting property of $\mathscr{C}$ and the exactness of $\otimes H$,
a sequence of morphisms in $\mathscr{C}^H$ is exact if and only if it is so, regarded as 
a sequence of morphisms in $\mathscr{C}$. 
Algebras in $\mathscr{C}^H$, which are supposed to be commutative, are called \emph{$H$-comodule algebras}
in $\mathscr{C}$. An obvious example of such algebras is $H=(H,~\Delta_H)$. 
One can define \emph{square-zero extensions} and \emph{Hochschild extensions} in $\mathscr{C}^H$, just
as those in $\mathscr{C}$. 
Our interest concentrates in 
such square-zero extensions of $H$ and Hochschild extensions
with cokernel $H$. 

Let $(E,p)$ be a square-zero extension of $H$ in $\mathscr{C}^H$.
As was seen above for square-zero extensions in $\mathscr{C}$, the kernel
$\operatorname{Ker} p$ of $p$ is naturally regarded as an $H$-module 
in $\mathscr{C}^H$. Here one sees that the category
\[
{}_H\mathscr{C}^H={}_H(\mathscr{C}^H)
\]
of $H$-modules in $\mathscr{C}^H$ is abelian. 
In view of \cite[Section 4.1]{Sw}, one may call an object of ${}_H\mathscr{C}^H$ an \emph{$H$-Hopf module}
in $\mathscr{C}$. 
Given an object
$X$ of $\mathscr{C}$, regard it as a \emph{trivial} $H$-module and $H$-comodule 
in $\mathscr{C}$; thus, $H$ acts by
$\varepsilon_H\otimes \operatorname{id}_X: H\otimes X \to \mathbf{1} \otimes X=X$, and coacts 
by
$\operatorname{id}_X \otimes u_H : X= X \otimes \mathbf{1} \to X \otimes H$. 
One sees that the tensor product 
\begin{equation}\label{EXH}
X \otimes H
\end{equation}
of $H$-modules and $H$-comodules in $\mathscr{C}$ is an object of ${}_H\mathscr{C}^H$.

\begin{lemma}\label{LHopfm}
$X \mapsto X \otimes H$ gives rise to an equivalence 
$\mathscr{C}\approx {}_H\mathscr{C}^H$ of abelian categories. 
\end{lemma}

This is a standard result, which can be proved essentially in the same way as proving the Hopf-module Theorem 
\cite[Theorem 4.1.1]{Sw} in $\mathsf{Vec}$. 

\begin{lemma}\label{Ladj}
$X\mapsto X\otimes H$ gives an additive functor $\mathscr{C}\to \mathscr{C}^H$ right adjoint to
the forgetful functor $\mathscr{C}^H\to \mathscr{C}$. 
\end{lemma}

This is also a standard result. Indeed, let $N=(N,~\rho_N)$ be an object of $\mathscr{C}^H$. To every
morphism $f : N \to X$ of $\mathscr{C}$, assign its $H$-colinear extension given by
\begin{equation}\label{Etilde}
\widetilde{f} : 
N \overset{\rho_N}{\longrightarrow} N \otimes H \overset{f \otimes \operatorname{id}_H}{\longrightarrow}
X \otimes H. 
\end{equation}
The assignment $f \mapsto \widetilde{f}$ gives a natural isomorphism 
\begin{equation}\label{ECNX}
\mathscr{C}(N,\ X) \overset{\simeq}{\longrightarrow} \mathscr{C}^H(N,\ X\otimes H)
\end{equation}
of abelian groups,
whose inverse assigns to every morphism $N\to X\otimes H$ in $\mathscr{C}^H$,
its composition with $\operatorname{id}_X \otimes \varepsilon_H : X \otimes H \to X \otimes\mathbf{1}
=X$. The result is restated by the lemma above. 

The corollary below now follows since $\mathscr{C}$ is split, whence every object of
$\mathscr{C}$ is injective. 

\begin{corollary}\label{Cinj}
Every object of ${}_H\mathscr{C}^H$, which is necessarily isomorphic to $X \otimes H$ (see 
\eqref{EXH}) for 
some object $X$ of $\mathscr{C}$, is injective, regarded as an object of $\mathscr{C}^H$. 
\end{corollary}

Fix an object $X$ of $\mathscr{C}$, which gives rise to
the object $X \otimes H$ of ${}_H\mathscr{C}^H$ as in \eqref{EXH}.

\begin{prop}\label{PHochExt}
We have the following. 
\begin{itemize}
\item[(1)]
Every Hochschild extension
\[
\langle E\rangle : 0 \to X \otimes H\overset{i}{\longrightarrow} 
E \overset{p}{\longrightarrow} H\to 0
\]
of $H$ by $X \otimes H$ in $\mathscr{C}^H$ splits in $\mathscr{C}^H$.
\item[(2)]
The set 
\[
\operatorname{Ext}(H,\ X\otimes H)_{\mathscr{C}^H}
\]
of all equivalence classes
of the Hochschild extensions of $H$ by $X\otimes H$ in $\mathscr{C}^H$
is in a natural one-to-one correspondence with the symmetric 2nd Hochschild cohomology
\[
H^2_s(H,\ X\otimes H)_{\mathscr{C}^H}
\]
constructed in $\mathscr{C}^H$. The set is also in one-to-one correspondence, 
through $\langle E\rangle \mapsto (E,~p)$, with
the set of all isomorphism classes of the square-zero extensions of $H$ in $\mathscr{C}^H$
with kernel isomorphic to $X\otimes H$. 
\end{itemize}
\end{prop}

\begin{proof}
(1)\ 
This is a direct consequence of Corollary \ref{Cinj}.

(2)\ 
By Part 1 every Hochschild extensions of $H$ by $X\otimes H$ splits as a short exact sequence in $\mathscr{C}^H$, whence it can be described, just as the $E$ in \eqref{EEMA} is done by \eqref{Eum}, by a symmetric Hochschild $2$-cocycle $H\otimes H\to X\otimes H$; see \eqref{Eumtilde} below. 
The standard argument which connects the equivalence of extensions 
with the cohomologous relation proves the first assertion. 
The second is easy to see. 
\end{proof}

\begin{prop}\label{PH2}
We have a natural isomorphism 
\[
H^2_s(H,\ X)_{\mathscr{C}}\overset{\simeq}{\longrightarrow} H^2_s(H,\ X\otimes H)_{\mathscr{C}^H}
\]
of abelian groups between 
the symmetric 2nd Hochschild cohomologies. 
\end{prop}

Before proving this we make some remarks. First, notice $\operatorname{Im}(u_H)\oplus H^+=H$, where 
$H^+=\operatorname{Ker}(\varepsilon_H : H \to \mathbf{1})$
denotes the augmentation ideal of $H$. Therefore, we have the canonical identification
\[
H^+= \operatorname{Coker}(u_H : \mathbf{1}\to H).
\]
Notice that $\operatorname{Coker}(u_H)$ is a quotient
object of $H$ in $\mathscr{C}^H$. We regard $H^+$ as 
an object of $\mathscr{C}^H$ through the identification above.

Second, given a Hochschild extension $\langle E\rangle$ of $H$
by $X$ (resp., by $X \otimes H$) in $\mathscr{C}$ 
(resp., in $\mathscr{C}^H$), 
we can and we do choose 
a section $j$ of $E \to H$ so as to be unit-preserving,  $j\circ u_H=u_E$, by replacing it with $j + (u_E-j\circ u_H)\circ \varepsilon_H$. Therefore, we have $u_E=(0, u_H)$
in the associated coproduct $E=X\oplus H$ (resp., $E=(X\otimes H)\oplus H$) such as the one in \eqref{EEMA}; 
cf. \eqref{Eum}. 
Accordingly, we suppose that the
cochains for computing the relevant symmetric 2nd Hochschild cohomologies are \emph{normalized}; they thus have
some tensor powers of $H^+$ as their domains. 
In addition, for the symmetry condition of the $2$-cochains we use
the 2nd symmetric power
$S^2(H^+)$
of $H^+\, (=\operatorname{Coker}(u_{H}))$ in $\mathscr{C}$ (resp., in $\mathscr{C}^H$). 
Explicitly, the 2nd cohomologies in question are obtained as the cohomologies at the middle terms
of the 
horizontal (so-called normalized Hochschild) complexes below. 
\begin{equation}\label{Ecomp}
\begin{xy}
(0,8)   *++{\mathscr{C}(H^+,\ X)}  ="1",
(39,8)  *++{\mathscr{C}(S^2(H^+),\ X)} ="2",
(80,8) *++{\mathscr{C}((H^+)^{\otimes 3},\ X)} ="3",
(0,-8) *++{\mathscr{C}^H(H^+,\ X\otimes H)} ="4",
(39,-8) *++{\mathscr{C}^H(S^2(H^+),\ X\otimes H)} ="5",
(80,-8) *++{\mathscr{C}^H((H^+)^{\otimes 3},\ X\otimes H)} ="6",
{"1" \SelectTips{cm}{} \ar @{->}^{\delta^1} "2"},
{"2" \SelectTips{cm}{} \ar @{->}^{\delta^2} "3"},
{"1" \SelectTips{cm}{} \ar @{->} "4"},
{"2" \SelectTips{cm}{} \ar @{->} "5"},
{"3" \SelectTips{cm}{} \ar @{->} "6"},
{"4" \SelectTips{cm}{} \ar @{->} "5"},
{"5" \SelectTips{cm}{} \ar @{->} "6"}.
\end{xy}
\end{equation}
Here the vertical arrows indicate the natural isomorphisms such as in \eqref{ECNX}. 

\begin{proof}[Proof of Proposition \ref{PH2}]
It remains to prove commutativity of the diagram above. But this can be verified
directly. Indeed, one may replace the $S^2(H^+)$ at the middle terms with $(H^+)^{\otimes 2}$,
so that the resulting complexes are semi-cosimplicial.
One then has only to verify the commutativity for the relevant coface operators, not for
the coboundary operators such as $\delta^i$ above. Let $i\in \{1,2\}$. 
A point of the verification is to use the fact that
the face operators 
\[
d_q : H\otimes (H^+)^{\otimes (i+1)}\otimes H \to H\otimes (H^+)^{\otimes i}\otimes H,\quad 0\le q \le i+1
\]
which induce the coface operators (of the two complexes) 
are morphisms in the category ${}_H(\mathscr{C}^H)_H$ of $(H, H)$-bimodules in
$\mathscr{C}^H$. 
\end{proof}

\begin{rem}\label{R1to1}
By Lemma \ref{LHopfm} and Propositions \ref{PHochExt} and \ref{PH2} we have
a natural one-to-one correspondence between the following two sets:
\begin{itemize}
\item all isomorphism classes of the square-zero extensions $(E,p)$ of $H$ in $\mathscr{C}$
such that the kernel $\operatorname{Ker}p$ of $p : E \to H$ is a trivial $H$-module;
\item all isomorphism classes of the square-zero extensions of $H$ in $\mathscr{C}^H$.
\end{itemize}
To present the correspondence explicitly, suppose that we are given $(E,p)$ from the first set.
We may suppose $E=X\oplus H$ with $X$ a trivial $H$-module,
and $p : X\oplus H \to H$ is the projection. In addition, 
the product $m_E$ of $E=X\oplus H$ is such as in \eqref{Eum}
with obvious modifications, where in particular, the symmetric Hochschild 2-cocycle 
$s$ is now supposed to be normalized, and thus the unit is $u_E=(0, u_H)$. 
The corresponding square-zero extension of $H$ in $\mathscr{C}^H$ is $\widetilde{E}=\widetilde{X}\oplus H$ 
together with the projection to $H$, where $\widetilde{X}=X\otimes H$ as in \eqref{EXH}, and 
the structure of $\widetilde{E}$ is given by
\begin{equation}\label{Eumtilde}
u_{\widetilde{E}}=(0,\ u_H),\quad 
m_{\widetilde{E}}=\begin{cases} (\widetilde{s},~m_H) &\text{on}\ H\otimes H;\\
((\mathrm{id}_X\otimes m_H)\circ(c_{H,X}\otimes \mathrm{id}_H),~0) &\text{on}\ H\otimes \widetilde{X};\\
(\mathrm{id}_X\otimes m_H,~0) &\text{on}\ \widetilde{X}\otimes H;\\
\ 0 &\text{on}\ \widetilde{X}\otimes \widetilde{X},
\end{cases}
\end{equation}
where $\widetilde{s} : H \otimes H \to \widetilde{X}=X\otimes H$ denotes the $H$-colinear extension
of $s$; see \eqref{Etilde}. The original $(E, p)$ is recovered from $\widetilde{E}$ by dividing it by the ideal
$X\otimes H^+\, (\subset \widetilde{X})$. 
\end{rem}

We wish to present $H^2_s(H,\ X)_{\mathscr{C}}$ in a simple form. 
Let $m_{H^+} : H^+ \otimes H^+ \to H^+$ denote the restricted product on $H^+$. Let
\[
\delta_1 : S^2(H^+)\to H^+
\]
denote the morphism of $\mathscr{C}$ induced by $m_{H^+}$, and let
\[
\delta_2 :  (H^+)^{\otimes 3} \to (H^+)^{\otimes 2} \to S^2(H^+)
\]
denote the composition of  $\operatorname{id}_{H^+}\otimes m_{H^+}- m_{H^+}\otimes\operatorname{id}_{H^+}$
with the canonical $(H^+)^{\otimes 2}\to S^2(H^+)$. Regard $H^+$ naturally as an object of the symmetric monoidal 
category ${}_H\mathscr{C}$ of $H$-modules in $\mathscr{C}$,
whose tensor product and unit object are now supposed to be $\otimes_H$ and $H$, respectively, 
and construct the
2nd symmetric power 
$S_H^2(H^+)$
of it; this, constructed in ${}_H\mathscr{C}$, should be distinguished from the one $S^2(H^+)$
constructed in $\mathscr{C}$, before. 
Notice that this is, in fact, the cokernel $\operatorname{Coker}(\delta_2)$ of $\delta_2$. 
Let
\[
\mu_H : S_H^2(H^+) \to H^+
\]
denote the morphism of $\mathscr{C}$ induced by $\delta_1$, and thus by $m_{H^+}$, too.  

\begin{prop}\label{PH2C}
We have the following.
\begin{itemize}
\item[(1)]
The coboundary operators $\delta^i$ of the upper horizontal complex in \eqref{Ecomp} for
computing $H^2_s(H,\ X)_{\mathscr{C}}$ are induced by $\delta_i$ so that
\[
\delta^i=\mathscr{C}(\delta_i,\ X),
\]
where $i=1, 2$.
\item[(2)]
We have an isomorphism
\begin{equation}\label{EH2Ker}
H^2_s(H,\ X)_{\mathscr{C}}\simeq \mathscr{C}(\operatorname{Ker}(\mu_H),\ X)
\end{equation}
of abelian groups, which is natural in $X$.
\end{itemize}
\end{prop}
\begin{proof}
Part 1 is easily verified by using the fact that $X$ is trivial. 
The result for $\delta^2$,
combined with the splitting property of $\mathscr{C}$, shows that
the object of $2$-cocycles is naturally isomorphic to $\mathscr{C}(S_H^2(H^+),\ X)$. 
This, combined with the result for $\delta^1$ in Part 1, implies Part 2. 
\end{proof}

Recall from the Introduction the conditions (a)--(d) for Hopf algebras in $\mathsf{Vec}$, ignoring (e) and (f). 
The conditions (a)--(d) make sense for Hopf algebras in $\mathscr{C}$, more generally. 
Indeed, the definition of (equivariant) smoothness is directly generalized. 
The morphism $\mu_H$ is already given in the generalized situation. 
For (c) some comments should be made. Due to the properties (I)--(II) of $\mathscr{C}$, given a Hopf algebra
$H$ in $\mathscr{C}$, a Hopf ideal of $H$ and thus a quotient Hopf algebra of $H$ are defined in an 
obvious manner. 
Suppose that $Q$ is a quotient Hopf algebra of $H$. Then $H$ is naturally regarded as a $Q$-comodule
on both sides. 
Given a right $Q$-comodule $V=(V, \rho_V)$ and a left $Q$-comodule $W=(W,\lambda_W)$ in $\mathscr{C}$, 
the co-tensor product $V\square_Q W$ is defined by the following equalizer diagram
\begin{equation}\label{Ecotens}
0\to V\, \square_Q W\to V\otimes W \rightrightarrows (V \otimes Q)\otimes W\, (= V \otimes (Q\otimes W))
\end{equation}
in $\mathscr{C}$, where the paired arrows indicate $\rho_V\otimes \mathrm{id}_W$ and $\mathrm{id}_V\otimes \lambda_W$. Indeed, this defines a functor, $V\, \square_Q : {}^Q\mathscr{C}\to\mathscr{C}$ (resp.,
$\square_QW : \mathscr{C}^Q \to\mathscr{C}$), which is seen to be left exact. 
If it is exact, we say that $V$ (resp., $W$) is \emph{coflat}. 
A right $Q$-comodule is regarded, with the coaction twisted by
the antipode of $Q$, as a left $Q$-comodule, and vice versa. 
The thus twisted left $Q$-comodule $H$ is isomorphic to the natural 
left $Q$-comodule $H$ through the antipode $\mathcal{S}_H$ of $H$.
Therefore, the co-tensor-product functors $H\, \square_Q$ and 
$\square_QH$ are identified through the natural isomorphisms induced 
from $(\mathrm{id}_W\otimes \mathcal{S}_H)\circ c_{H,W} : H \otimes W
\overset{\simeq}{\longrightarrow} W \otimes H$. 
As a conclusion, $H$ is coflat as a right $Q$-comodule
if and only if it is so as a left $Q$-comodule. 

\begin{rem}\label{Rcoflatinj}
The above definitions of the co-tensor product and the coflatness (for $Q$) are valid for any coalgebra
in $\mathscr{C}$. Let $C$ be such a coalgebra. A $C$-comodule $V=(V,\rho_V)$, say on the right, is coflat if it is injective,
as is seen from the canonical isomorphism $(V\otimes C)\square_CW=V\otimes W$, where 
$W$ is a left $C$-comodule, and the fact that
$\rho_V$ splits $C$-colinearly provided $V$ is injective. The converse holds when $\mathscr{C}=\mathsf{Vec}$,\
$\mathsf{sVec}$ or $\mathsf{Ver}_p^{\mathrm{ind}}$; see \eqref{Ecat}. 
Indeed, the proof of the well-known result \cite[Proposition A.2.1]{Tak1} in the case where $\mathscr{C}=\mathsf{Vec}$ works in the other cases, as well. 
For the following properties, which are essential to the cited proof, are verified 
in those cases: (i)~every $C$-comodule is ind-finite (for so is $C$), and (ii)~we have a
natural isomorphism 
\[
\mathscr{C}^C(V,W)\simeq V \square_CW^*,
\]
where $W$ is supposed to have finite length, so that the dual object $W^*$ is naturally a left $C$-comodule.
\end{rem}

We make some additional remarks before the proof of Theorem \ref{T2nd}. 
The category of algebras in $\mathscr{C}$ has pull-backs. Indeed, the 
pull-back $L\times_SR$ of two algebra morphisms $L \to S \leftarrow R$ in $\mathscr{C}$ is, as an object
of $\mathscr{C}$, 
the pull-back of those regarded as morphisms of $\mathscr{C}$,
and is equipped with the component-wise algebra structure. Given an algebra $R$ in $\mathscr{C}$ 
with a square-zero ideal $I$ and an algebra morphism $f : H \to R/I$, the pull-back 
$T:=H\times_{R/I}R$ actually constitutes the diagram
\begin{equation}\label{Epb}
\begin{xy}
(0,6)   *++{T}  ="1",
(14,6)  *++{R} ="2",
(0,-7) *++{H} ="3",
(14,-7) *++{R/I.} ="4",
{"1" \SelectTips{cm}{} \ar @{->}^{\omega} "2"},
{"1" \SelectTips{cm}{} \ar @{->}_{\varpi} "3"},
{"2" \SelectTips{cm}{} \ar @{->} "4"},
{"3" \SelectTips{cm}{} \ar @{->}_f "4"}
\end{xy}
\end{equation}
One sees that $(T, \varpi)$ is a square-zero extension of $H$ in $\mathscr{C}$.
If $\varpi$ has a section, then it, composed with $\omega$,
gives a lift of $f$. This familiar argument shows that Condition (a) is equivalent to saying that for every square-zero extension 
$(E,p)$ of $H$ in $\mathscr{C}$, 
$p$ splits as algebra morphism in $\mathscr{C}$; see \cite[Section 9.3.2]{Wei}, for example. 
This in turn is equivalent to
\begin{itemize}
\item[(a$'$)] $H_s^2(H, M)_{\mathscr{C}}=0$ for every $H$-module $M$ in $\mathscr{C}$. 
\end{itemize}
See the paragraph containing \eqref{EH2AM}. 

Similarly, the category of algebras in $\mathscr{C}^H$ has pull-backs.
The argument above modified into $\mathscr{C}^H$, combined with 
Propositions \ref{PHochExt} and \ref{PH2}, shows that Condition (b)
is equivalent to
\begin{itemize}
\item[(b$'$)] $H_s^2(H, X)_{\mathscr{C}}=0$ for every object $X$ of $\mathscr{C}$, regarded as a trivial $H$-module.
\end{itemize}

\begin{proof}[Proof of Theorem \ref{T2nd}]
(c)$\Rightarrow$(a).\ 
This follows, since $H/H^+\, (=\mathbf{1})$ is a quotient Hopf algebra of $H$, over which $H$ is obviously coflat, 
and being smooth is the same as being $H/H^+$-smooth. 

(a)$\Rightarrow$(d).\ 
By Proposition \ref{PH2C} (2), Condition (d) is equivalent to (b$'$), which is satisfied under (a$'$),
or equivalently, under (a). 

(d)$\Rightarrow$(b).\
In fact, the conditions are equivalent, since they are both equivalent to (b$'$).

(b)$\Rightarrow$(c).\ 
This follows by applying the next proposition for $H$-comodule algebras in $\mathscr{C}$,
to $H$, in particular. 
\end{proof}

\begin{prop}[\text{see \cite[Proposition 3.19]{MOT}}]\label{PMOTeq}
Let $\mathscr{C}$ be as above, and let $H$ be a Hopf algebra in $\mathscr{C}$. 
If $Q$ is a quotient Hopf algebra of $H$ such that
$H$ is coflat as a right (or equivalently, left) $Q$-comodule, then every $H$-smooth 
$H$-comodule algebra in $\mathscr{C}$ is necessarily $Q$-smooth. 
\end{prop}
\begin{proof}
Let $Q$ be a quotient Hopf algebra of $H$. 
Given an object $X$ of $\mathscr{C}^Q$ and $Y$ of $\mathscr{C}$, 
we have a canonical isomorphism 
$(X \, \square_Q H)\otimes Y =X \, \square_Q (H\otimes Y)$,
since $\otimes Y$ is exact. 
It follows that $X \, \square_QH$ naturally turns into an
object of $\mathscr{C}^H$. In fact, the co-tensor product $\square_Q H$ 
gives rise to a functor $\mathscr{C}^Q\to \mathscr{C}^H$.  
Moreover, we
have a natural isomorphism of abelian groups, 
\[
\mathscr{C}^Q(N,\ X)\overset{\simeq}{\longrightarrow} \mathscr{C}^H(N,\ X\, \square_Q H), 
\]
analogous to (in fact, generalizing) \eqref{ECNX}; $f$ from the group on the left maps to 
$(f\, \square_Q\operatorname{id_H})\circ \rho_N : N \to X\, \square_QH$ on the right. 
Since $\square_Q H$ is a lax monoidal functor, for which the relevant constraints are not necessarily isomorphic by definition, 
it follows that if $X$ is an algebra in $\mathscr{C}^Q$, then $X\, \square_Q H$ is such in $\mathscr{C}^H$.
Moreover, if $N$ is an algebra in $\mathscr{C}^H$, as well, then the isomorphism above restricts to 
a bijection  
\begin{equation*}\label{Ecolinext}
\mathrm{Alg}\mathscr{C}^Q(N,\ X)\overset{\simeq}{\longrightarrow} \mathrm{Alg}\mathscr{C}^H(N,\ X\, \square_Q H),
\end{equation*}
where $\mathrm{Alg}\mathscr{C}^Q$ and $\mathrm{Alg}\mathscr{C}^H$ denote the categories of
the algebras in respective categories. 

Let $A$ be an $H$-comodule algebra in $\mathscr{C}$, or namely, an algebra in $\mathscr{C}^H$. 
Given an algebra $R$ in $\mathscr{C}^Q$ and its nilpotent ideal $I$, we have the natural commutative diagram
\[
\begin{xy}
(0,7)*++{\mathrm{Alg}\mathscr{C}^Q(A,\ R)} ="1",
(41,7)  *++{\mathrm{Alg}\mathscr{C}^H(A,\ R\, \square_Q H)} ="2",
(0,-8)   *++{\mathrm{Alg}\mathscr{C}^Q(A,\ R/I)}  ="3",
(41,-8)  *++{\mathrm{Alg}\mathscr{C}^H(A,\ R/I\, \square_Q H).} ="4",
{"1" \SelectTips{cm}{} \ar @{->}^{\hspace{-3mm}\simeq} "2"},
{"3" \SelectTips{cm}{} \ar @{->}^{\hspace{-3mm}\simeq} "4"},
{"1" \SelectTips{cm}{} \ar @{->} "3"},
{"2" \SelectTips{cm}{} \ar @{->} "4"}
\end{xy}
\]
Assume that $H$ is $Q$-coflat. Then the natural algebra
morphism $R \, \square_Q H \to R/I\, \square_Q H$
in $\mathscr{C}^H$ is epic, whose kernel, $I\, \square_QH$,
is nilpotent. 
If $A$ is $H$-smooth, then the vertical arrow on the right is surjective, whence so is the one on the left. 
This shows that $A$ is $Q$-smooth. 
\end{proof}

\begin{rem}\label{Rnoncom}
It is now easy to show a non-commutative analogue of Theorem \ref{T2nd} in view of the proof of
\cite[Theorem 1.2]{MO}. 
Suppose that $\mathscr{C}$ is a semisimple abelian braided monoidal category which satisfies the
same conditions as (I)--(II), and $H$ is a Hopf algebra in $\mathscr{C}$, in general. 
The definition of \emph{(equivariant) smoothness} extends to 
this context in an obvious manner (though one might prefer the word ``quasi-free" to ``smooth", in view of \cite[Section 9.3.2]{Wei}), and we see that the following are equivalent: 
\begin{itemize}
\item[($\alpha$)] $H$ is smooth;
\item[($\beta$)] $H$ is $H$-smooth as a right $H$-comodule algebra;
\item[($\gamma$)] $H$ is $Q$-smooth as a right $Q$-comodule algebra
for every quotient Hopf algebra $Q$ of $H$ such that $H$ is coflat as a left
$Q$-comodule.
\item[($\delta$)] $H$ is left hereditary as an algebra in $\mathscr{C}$, or in other words, 
an $H$-submodule of any left $H$-module that is projective is necessarily projective. 
\end{itemize}
 
Indeed, a point of the proof is to see the following: as an analogue of Condition (b$'$) above, the condition
\begin{itemize}
\item[($\beta'$)] For every left $H$-module $M$ in $\mathscr{C}$, the 2nd Hochschild
cohomology $H^2(H, M)_{\mathscr{C}}$ constructed in $\mathscr{C}$ vanishes, where the $M$ of coefficients is regarded as an $(H,H)$-bimodule
with respect to the trivial right $H$-action through the counit $\varepsilon_H$
\end{itemize}
is equivalent to ($\beta$). 
Here, $H^2(H,M)_{\mathscr{C}}$ is naturally isomorphic to the 2nd Ext group
$\operatorname{Ext}^2_H(\mathbf{1},M)_{\mathscr{C}}$ for the left $H$-modules in 
$\mathscr{C}$, which form an abelian category with enough projectives. In addition, given 
a left $H$-module $N$ in $\mathscr{C}$, 
a projective resolution of the trivial $H$-module $\mathbf{1}$, with $\otimes N$ applied, turns into
a projective resolution of $N$. 

The conditions above, except ($\alpha$), have the variants with the sides, left and right, exchanged, 
and those variants and ($\alpha$) (thus, the above conditions all together) are proved to be equivalent to each other. 
\end{rem}


\section{Proof of Theorem \ref{T1st}}\label{Sec3}

We work in the category $\mathsf{Vec}$ of vector spaces
over a fixed field $k$,
and let $H$ be an ordinary Hopf algebra over $k$. 
The coproduct, the counit and the antipode of $H$ are denoted by
$\Delta_H$, $\varepsilon_H$ and $\mathcal{S}_H$, respectively, as before. The coproduct 
will be presented so that 
\[
\Delta_H(h)=h_{(1)}\otimes h_{(2)},\quad h \in H.
\]
By Theorem \ref{T2nd}, Conditions (a)--(d) for $H$ are known to be equivalent.

\subsection{Proof of Part 1, characteristic zero case}\label{subsec:charzero}
We do not assume $\operatorname{char} k =0$ until explicitly mentioned. 

The following is well known. 
But we give a proof for the sake of further discussion. 

\begin{lemma}\label{LSmGr}
Suppose that $H$ is of finite type. Then Conditions (a) and (e) are equivalent, 
or explicitly, $H$ is smooth if and only if
it is geometrically reduced. 
\end{lemma}
\begin{proof}
It is known that a Noetherian algebra is smooth if and only if it is geometrically
regular, i.e., it remains regular after base extension to the algebraic closure of $k$; see 
\cite[Corollary 9.3.13 and the following Remark]{Wei}. 

\emph{`Only if'.}\ The geometrical regularity implies the geometrical reducedness. 

\emph{`If'.}\ 
The $H$-module $\Omega_{H/k}$ of K\"{a}hler differentials 
is projective (in fact, free); see \cite[11.3 Theorem]{Wa}, for example. 
In general, a finite-type algebra $A$ with projective 
K\"{a}hler differential module is smooth, or equivalently, geometrically regular, either if $\operatorname{char} k=0$ or if
$\operatorname{char} k>0$, $k$ is perfect
and $A$ is reduced; see \cite[Lemmas 33.25.1 and 33.25.2]{Stack}. 
In the case where $\operatorname{char} k>0$, apply this to the base extension of $H$ to the algebraic closure
of $k$.
\end{proof}

Recall that $H$ can be presented as a directed union 
\begin{equation}\label{Eunion}
H=\bigcup_{\alpha}H_{\alpha}
\end{equation}
of Hopf subalgebras $H_{\alpha}$ of finite type. 
Accordingly, we have
\begin{equation}\label{Eindlim}
S_H^2(H^+)=\varinjlim_{\alpha}S_{H_{\alpha}}^2(H_{\alpha}^+),\quad \mu_H=\varinjlim_{\alpha}\mu_{H_{\alpha}}. 
\end{equation}

Assume $\operatorname{char} k=0$. Notice from the general fact noted in the last paragraph of the last proof  
that every $H_{\alpha}$ satisfies Condition (a) and the equivalent (e), whence $H$ does (e). 
On the other hand, Theorem \ref{T2nd} and Lemma \ref{LSmGr} ensure the equivalence
of Conditions (a)--(e) for every $H_{\alpha}$ (in arbitrary characteristic). 
In view of \eqref{Eindlim}, it follows that $H$ satisfies (d), and thus (a)--(c), as well, by the cited theorem. 
This proves Part 1 of Theorem \ref{T1st}. 

\subsection{Proof of Part 2, positive characteristic case}\label{subsec:charp}
Throughout in this subsection we assume $\operatorname{char}k=p>0$.

A main subject of our argument is the abelian group
$H_s^2(H,~k)$, which is now a vector space. The coefficients are in the 
trivial $H$-module $k$.
By \eqref{EH2Ker} we have a natural $k$-linear isomorphism
\begin{equation}\label{EdualKer}
H_s^2(H,\ k) \simeq \operatorname{Hom}(\operatorname{Ker}(\mu_H),\ k).
\end{equation}

\begin{prop}\label{Pbpp}
Conditions (a)--(d), which are known to be mutually equivalent by Theorem \ref{T2nd}, are equivalent to 
\begin{itemize}
\item[(b$''$)] $H_s^2(H,~k)=0$, where $k$ is regarded as a trivial $H$-module. 
\end{itemize}
\end{prop}
\begin{proof}
This holds since we have by \eqref{EdualKer} that (b$''$)$\Leftrightarrow$(d). 
\end{proof}

To prove Part 2 of Theorem \ref{T1st}, notice first that (e) implies (f). 
Takeuchi \cite[Proposition 1.9]{T} proves that
(f) implies (d). 
In view of Proposition \ref{Pbpp}, it remains to prove that (b$''$) implies (e). 
This follows from the following proposition, which, indeed, shows that if $H_s^2(H,~k)=0$, then
we have $H_s^2(J,~k)=0$ for every finite-type Hopf subalgebra $J$ of $H$. 
By Proposition \ref{Pbpp} now applied to $J$, 
the last condition is equivalent to 
every $J$ being smooth. This is equivalent, by Lemma \ref{LSmGr}, to every $J$ being geometrically reduced, 
or to Condition (e).

\begin{prop}\label{Pres}
Given a Hopf subalgebra $J$ of $H$, the restriction map
\[
\operatorname{res} : H_s^2(H, k) \to H_s^2(J, k)
\]
is surjective.
\end{prop}
\begin{proof}
Let us consider the map
\begin{equation}\label{ESS}
S_J^2(J^+)\to S_H^2(H^+)
\end{equation}
which arises naturally from the inclusion $J \hookrightarrow H$. The five lemma
shows that
the map above is injective 
if and only if so is its restriction $\operatorname{Ker}(\mu_J)\to \operatorname{Ker}(\mu_H)$ to the kernels
of the $\mu$ maps. 
It follows by the isomorphism \eqref{EdualKer} and its naturality that the injectivity 
of the map \eqref{ESS}
is equivalent to the desired surjectivity of $\operatorname{res}$. 

Present $H=\bigcup_{\alpha}H_{\alpha}$ as in \eqref{Eunion}. 
To prove the injectivity of the map \eqref{ESS} (and the equivalent surjectivity of $\operatorname{res}$), we may (and we do)
suppose
that $H$ and thus $J$ are of finite type, in view of \eqref{Eindlim}, replacing 
$H$ and $J$ with $H_{\alpha}$ and $J \cap H_{\alpha}$, respectively. 
Here we remark that $J \cap H_{\alpha}$ is of finite type by \cite[Corollary 3.11]{Tak}. 

In general, for an algebra $A$, the symmetric 2nd Hochschild cohomology
$H_s^2(A,~-)$ is identified, through a natural isomorphism,
with
the 1st Andr\'{e}-Quillen cohomology $D^1(A\mid k,~-)$;
see \cite[Exercise 8.8.4]{Wei}. The restriction map
for $H^2_s$ is identified with such a map for $D^1$, 
which extends up to the 2nd Andr\'{e}-Quillen cohomology so that
\[
D^1(A\mid k,~M) \overset{\operatorname{res}}{\longrightarrow} D^1(B\mid k,~M)\longrightarrow D^2(A\mid B,~M),
\]
where $B$ is a subalgebra of $A$, and $M$ is an $A$-module.

Therefore, for our purpose, it suffices to prove 
that $D^2(H\mid J,~k)=0$. 
Let $Q:=H/J^+H$ be the quotient Hopf algebra 
of $H$ corresponding $J$; this $Q$ is necessarily of finite type.
We have 
an $H$-algebra isomorphism (see \cite[p.456, line 4]{T0})
\begin{equation}\label{EHJQ}
H \otimes_J H \overset{\simeq}{\longrightarrow} H \otimes Q,\quad
a \otimes_J b \mapsto ab_{(1)}\otimes (b_{(2)}\operatorname{mod}(J^+)),
\end{equation}
where $H\otimes_J H$ is supposed to be an $H$-algebra
through the embedding of $H$ into the left tensor factor. 
The base change property of the 
Andr\'{e}-Quillen cohomology (see \cite[Proposition 1.4.3]{MR}, for example), applied to the flat 
ring maps $J \to H$ and $k \to H$, shows 
\begin{equation*}\label{Ebchange}
\begin{split}
D^2(H\mid J,~M)&=D^2(H\otimes_J H\mid H,~M)\\
&=D^2(H\otimes Q\mid H,~M)= D^2(Q\mid k,~M),
\end{split}
\end{equation*}
where $M$ is a module over $H\otimes_J H\, (=H\otimes Q)$, and may be now supposed
to be the trivial module $k$. 
But this cohomology, in the case where $M=k$, vanishes, as desired. Indeed, since $Q$
is locally complete intersection by Corollary \ref{Clci} below, 
it follows by
Avramov's Theorem \cite[(1.2)]{A} that 
the 2nd Andr\'{e}-Quillen homology $D_2(Q\mid k,~k)$ vanishes,
whence
\[
D^2(Q\mid k,~k)=\operatorname{Hom}_k(D_2(Q\mid k,~k),~k)=0.
\]
See \cite[Exercise 8.8.6]{Wei}. 
\end{proof}


Suppose that $k$ is a perfect field (of $\operatorname{char}k=p>0$), and
let $H$ be a Hopf algebra of finite type. Then the nil radical $\operatorname{Nil} H$ of $H$
is a Hopf ideal, and the resulting quotient Hopf algebra
\[
H_{\operatorname{red}}:=H/\operatorname{Nil}H
\]
is geometrically reduced, or equivalently (see Lemma \ref{LSmGr}), ($H_{\operatorname{red}}$-)smooth.

\begin{prop}[\text{\cite[III, $\S$ 3, 6.4 Corollaire]{DG}}]\label{Plci}
Let the situation be as above. If $H\ne H_{\operatorname{red}}$, then $H$ is, 
regarded naturally as an $H_{\operatorname{red}}$-comodule algebra, in the form
\[
H=k[T_1,\dots, T_n]/(T_1^{p^{e_1}},\dots,
T_n^{p^{e_n}})
\otimes H_{\operatorname{red}},
\]
where $n$ and $e_i$, $1\le i \le n$, are positive integers.
\end{prop}

The corollary in \cite{DG} referred to above shows the result more generally for locally algebraic group schemes.
We will give below
a highly self-contained proof in the present affinity situation, which uses the
$H_{\operatorname{red}}$-smoothness of $H_{\operatorname{red}}$.

Let $A$ be an algebra over an arbitrary field, which is 
essentially of finite type, i.e., a localization of a finite-type algebra
by some multiplicatively closed subset.
We say that $A$ is \emph{locally complete intersection}, 
if the localization $A_{\mathfrak{p}}$ at every prime ideal $\mathfrak{p}$ is complete
intersection. A local algebra essentially of finite type,
such as $A_{\mathfrak{p}}$ above, is said to be \emph{complete
intersection}, if it is in the form
\begin{equation*}\label{EUkT}
S^{-1}k[T_1,\dots, T_n]/I,
\end{equation*}
where $S^{-1}k[T_1,\dots, T_n]$ is a localization of 
a polynomial algebra $k[T_1,\dots, T_n]$, and $I$
is an ideal generated by a regular sequence. 
It is known that $A$ is locally complete intersection if it is regular, or in particular, smooth. 
In addition, $A$ is locally complete intersection if it turns into such after some base field extension; see \cite[(5.11) (2)]{A}. 
These facts, combined with Proposition \ref{Plci}, prove
the following corollary in the non-trivial, positive-characteristic case.

\begin{corollary}\label{Clci}
A finite-type Hopf algebra over an arbitrary field is locally complete intersection. 
\end{corollary}

\begin{proof}[Proof of Proposition \ref{Plci}]
Let us be in the situation of the proposition.

Notice that the nil radical $\operatorname{Nil}H$ of $H$ is nilpotent. Since 
$H_{\operatorname{red}}$ is $H_{\operatorname{red}}$-smooth, 
the natural map $\pi : H \to H_{\operatorname{red}}=H/\operatorname{Nil}H$, 
as an $H_{\operatorname{red}}$-colinear algebra-surjection,   
has a section, say, $\phi : H_{\operatorname{red}}\to H$. 
Replacing $\phi$ with $a \mapsto \varepsilon_H(\phi(\mathcal{S}_{H_{\operatorname{red}}}(a_{(1)}))\phi(a_{(2)})$,
we may and we do suppose
that it is counit-preserving, $\phi \circ \varepsilon_{H_{\operatorname{red}}}=\varepsilon_H$.
Denote the left coideal subalgebra of $H_{\operatorname{red}}$-{\em{coinvariants}} in $H$ by
\begin{equation*}\label{EB}
B := H^{\operatorname{co}(H_{\operatorname{red}})}\, 
(=\{\, h \in H\mid h_{(1)}\otimes \pi(h_{(2)})=h\otimes \pi(1)\, \}),
\end{equation*}
and set $\varepsilon_B:=\varepsilon_H|_B$ and $B^+:=\operatorname{Ker}(\varepsilon_B)\, (=B\cap H^+)$. 
Notice that $\phi$ makes the $H_{\operatorname{red}}$-comodule (algebra) $H$ into 
an $H_{\operatorname{red}}$-Hopf module \cite[Section 4.1]{Sw}. By
the Hopf-module Theorem \cite[Theorem 4.1.1]{Sw} we see that
\[
\alpha : B \otimes H_{\operatorname{red}} \to H,\quad \alpha(b\otimes a)=b\phi(a)
\]
is an $H_{\operatorname{red}}$-colinear isomorphism of $B$-algebras, which is counit-preserving,
$\varepsilon_H\circ \alpha = \varepsilon_B\otimes \varepsilon_{H_{\operatorname{red}}}$. 
It follows that $B^+H=\operatorname{Nil}H$, whence $B$ is finite-dimensional and local. 

The rest is a modification of the proof of the Theorem in \cite[Section 14.4]{Wa}, which proves the result
in the special case where $B=H$. 
We prove by induction on $\dim B$ that $B$ is in the form
\begin{equation*}\label{EkTpe}
k[T_1,\dots,T_n]/(T_1^{p^{e_1}},\dots, T_n^{p^{e_n}}),
\end{equation*}
which, we understand, presents $k$ in the first induction step where $\dim B=1$. 
Recall that for an arbitrary algebra $A$, the
the Frobenius map $F_A : k^{1/p}\otimes A \to A$, $F_A(c^{1/p}\otimes a)=ca^p$ is defined, and it is 
a Hopf algebra map in case $A$ is a Hopf algebra. Since $k$ is perfect, we have
\[
\operatorname{Im}(F_A)=A^p\, (:=\{\, a^p\mid a \in A\, \}). 
\]
The isomorphism $\alpha$ restricts to
\[
\alpha^p : B^p \otimes (H_{\operatorname{red}})^p
\overset{\simeq}{\longrightarrow}H^p,
\]
which shows $(H_{\operatorname{red}})^p=(H^p)_{\operatorname{red}}$. 
In addition, $H^p$ is a Hopf subalgebra of $H$, and $\alpha^p$ is a couint-preserving
$(H^p)_{\operatorname{red}}$-colinear isomorphism of $B^p$-algebras, whence 
$B^p=(H^p)^{\operatorname{co}((H^p)_{\operatorname{red}})}$. 
We may suppose $\dim B>1$, which implies $\dim B^p<\dim B$. The induction hypothesis applied to
$B^p$ shows that there exist elements $y_1,\dots, y_s$ of $B^+$ and some positive integers $e_i$ such that
the algebra map
\begin{equation}\label{EisoBp}
k[T^p_1,\dots,T^p_s]/(T_1^{p^{e_1}},\dots, T_s^{p^{e_s}})\to B^p
\end{equation}
induced by $T_i^p\mapsto y_i^p$, $1 \le i \le s$, 
is an isomorphism. Notice that $y_i^{p^{e_i}}=0$, $1\le i\le s$. 
In addition, $y_1,\dots, y_s$ are linearly independent modulo $(B^+)^2$, since
they map through $F_B$ to $y_1^p,\dots, y_s^p$ in $(B^p)^+$ which are linearly independent modulo $((B^p)^+)^2$.
Choose elements $z_{s+1},\dots, z_n$ of $B^+$ which are maximal with respect to the properties
that $z_j^p=0$, $s<  j \le n$, and 
$z_{s+1},\dots, z_n$ are linearly independent modulo $(B^+)^2$. 
It may happen that $\{\, y_i\, \}$ or $\{\, z_j\, \}$ is empty. 

Notice that the assignment
\begin{equation}\label{Eassign}
T_i\mapsto \begin{cases}\ y_i, &1\le i\le s\\ \ z_i, &s<i \le n\end{cases}
\end{equation}
gives rise to an algebra map
\[
k[T_1,\dots,T_n]/(T_1^{p^{e_1}},\dots, T_s^{p^{e_s}}, T_{s+1}^p,\dots, T_n^p)\to B. 
\]
To complete the proof we wish to prove that this is an isomorphism.
Let 
\[
C:= B/(B^p)^+B. 
\]
This is finite-dimensional and local with the maximal ideal $C^+:=B^+/(B^p)^+B$.
The algebra map above, divided by the maximal ideals of the mutually isomorphic
local subalgebras in \eqref{EisoBp}, reduces to
\begin{equation}\label{EtoC}
k[T_1,\dots, T_n]/(T_1^p,\dots, T_n^p)\to C,
\end{equation}
which is induced again by \eqref{Eassign}. 
Since $B$ and $B^p$ are left coideal subalgebras of $H$, it follows by \cite[3.4 Theorem]{MW} that
$H$ is flat, and hence free, over $B$ and over $B^p$. It follows that $B$ is free over $B^p$.
By Nakayama's Lemma we have only to prove that the local algebra map \eqref{EtoC} is an 
isomorphism. Moreover, it suffices to prove that the associated graded algebra map
\begin{equation}\label{EtogrC}
k[T_1,\dots, T_n]/(T_1^p,\dots, T_n^p)\to \operatorname{gr} C
\end{equation}
is an isomorphism, where we suppose $\deg T_i=1$, $1\le i \le n$, and let
\[
\operatorname{gr}C:=\bigoplus_{n=0}^\infty(C^+)^n/(C^+)^{n+1}.
\]
For later use notice from $(B^p)^+B\subset (B^+)^2$ that 
\begin{equation}\label{ECB}
C^+/(C^+)^2=(B^++(B^p)^+B)/((B^+)^2+(B^p)^+B)=B^+/(B^+)^2.
\end{equation}

To prove that the graded algebra map \eqref{EtogrC} is isomorphic, we make 

\begin{claim}
$\operatorname{gr} C$ is, as a graded algebra, in the form $k[T_1,\dots, T_m]/(T_1^p,\dots, T_m^p)$.
\end{claim}

We postpone for a little while to prove the claim. 
By the claim we have only to prove that the map
\eqref{EtogrC} restricted to degree $1$ is isomorphic,
or in other words (in view of \eqref{ECB}), $y_1,\dots, y_s, z_{s+1},\dots, z_n$ form a basis 
of $B^+$ modulo $(B^+)^2$.

To see that they span $B^+$ modulo $(B^+)^2$, let $b \in B^+$. 
Since $b^p$ is a polynomial in the $y_i^p$ with zero constant, and $k$ is perfect,
we have a polynomial $u$ in the $y_i$ with zero constant, such that $b^p=u^p$. 
It follows by choice of $\{ z_j \}$ that $b-u\, (\in B^+)$ is a linear combination of
the $z_j$ modulo $(B^+)^2$, whence $b$ is such of the $y_i$ and the $z_j$. 
To see the linear independence, suppose that a linear combination
$\sum c_iy_i + \sum d_jz_j$ is in $(B^+)^2$.  
Since $z_j^p=0$, we have $\sum c_i^py_i^p$ is in $((B^p)^+)^2$, whence $c_i^p=0$, $1\le i\le s$. This implies
that the coefficients $c_i$ and $d_j$ are all zero. 

It remains to prove the claim. Define quotient Hopf algebras of $H_{\operatorname{red}}$ and of $H$ by
\[
K:=H_{\operatorname{red}}/(H_{\operatorname{red}}^p)^+H_{\operatorname{red}},\quad
J:=H/(H^p)^+H. 
\]
Notice that $K$ is a quotient Hopf algebra of $J$, and let $\mathfrak{a}$ be such that $J/\mathfrak{a}=K$.  
The isomorphisms $\alpha$ and $\alpha^p$ induce a counit-preserving $K$-colinear algebra-isomorphism
\[
C \otimes K \overset{\simeq}{\longrightarrow} J,
\]
from which we see $C=J^{\operatorname{co} K}$ and $C^+J=\mathfrak{a}$.
Construct
\[
\operatorname{gr}J=\bigoplus_{n=0}^\infty\mathfrak{a}^n/\mathfrak{a}^{n+1}.
\]
This is a commutative graded Hopf algebra with neutral component
$(\operatorname{gr}J)(0)=K$. 
The last isomorphism induces a $K$-colinear isomorphism
\[
\operatorname{gr} C \otimes K \overset{\simeq}{\longrightarrow} \operatorname{gr} J
\]
of graded algebras over $\operatorname{gr} C$. 
It follows that $\operatorname{gr} C$ turns into a graded Hopf algebra which is naturally 
isomorphic to the quotient graded Hopf algebra
\[
(\operatorname{gr}J)/K^+(\operatorname{gr}J)
\]
of $\operatorname{gr}J$. This has the properties: (i)~$(\operatorname{gr} C)(0)=k$,\
(ii)~$\operatorname{gr} C$ is generated by $(\operatorname{gr} C)(1)$,\ 
(iii)~the $p$-th power of every element of $(\operatorname{gr} C)(1)$, that is necessarily a primitive
element by (i), is zero. 
Hence we have a graded Hopf algebra surjection 
\[
k[T_1,\dots, T_m]/(T_1^p,\dots, T_m^p)\to \operatorname{gr} C
\]
which is an isomorphism in degree $1$. Here the $T_i$ are supposed to be primitive elements
of degree $1$. All the primitive elements of the source are of degree $1$, and they map isomorphically 
into $\operatorname{gr} C$. 
It follows by \cite[Lemma 11.0.1]{Sw} that the surjection above is injective, and is thus isomorphic, as desired.
\end{proof}

\section{Sample computations of $H_s^2(H,k)$}\label{Sec4}

This section is devoted to sample computations of the $H^2_s(H,k)$ which was the main subject of 
the argument proving Theorem \ref{T1st} (2). 
We thus work in $\mathsf{Vec}$, and let $H$ be a Hopf algebra over a field $k$, as in the preceding section.

\subsection{Augmented cleft extensions}\label{subsec:4.1}
We review from \cite{M3} a trick of the computation, modifying it into the commutative situation.
Here, $\operatorname{char}k$ may be arbitrary.

Fix a non-zero algebra $R$. 
Let $A=(A,\rho_A)$ be an \emph{$H$-cleft extension} \cite[Definition 7.2.1]{Mon} over $R$. 
Thus, $A$ is an $H$-comodule algebra such that
\[
R=A^{\operatorname{co}H}\, (:=\{ \, x \in A\mid \rho(x)=x\otimes 1\, \}),
\]
and there exists an $H$-colinear map $\phi : H\to A$
which is invertible with respect to the convolution product. We may and we do suppose
that such a $\phi$ is unit-preserving, $\phi(1_H)=1_A$, replacing the original $\phi$ 
with $h \mapsto \phi^{-1}(1_H)\phi(h)$,\ $H \to A$.
We have an $R$-linear and $H$-colinear isomorphism,
\begin{equation}\label{ERHA}
R \otimes H \to A,\quad x \otimes h \mapsto x\phi(h).
\end{equation}
This turns into an $H$-colinear $R$-algebra isomorphism, 
if we regard the source $R \otimes H$ as the $H$-crossed product $R \otimes_{\sigma} H$
which is equipped with the product 
\[
(x\otimes h) (y\otimes \ell):=xy\, \sigma(h_{(1)}, \ell_{(1)})\otimes h_{(2)}\ell_{(2)},
\quad x,y \in R,\ \, h, \ell \in H,
\]
given by the (invertible) symmetric Hopf 2-cocycle
\begin{equation*}\label{Esigma}
\sigma : H \otimes H \to R,\quad \sigma(h, \ell):=
\phi(h_{(1)})\phi(\ell_{(1)})\phi^{-1}(h_{(2)}\ell_{(2)});
\end{equation*}
see \cite[Section 7.1, Proposition 7.2.3]{Mon}. 
The unit of $R \otimes_{\sigma} H$ is $1 \otimes 1$
by the property $\phi(1_H)=1_A$. 
By saying $\sigma$ is \emph{symmetric}, we mean that $\sigma(h,\ell)=\sigma(\ell, h)$
for all $h,\ell \in H$.

Suppose that $R=(R,\epsilon_R)$ is augmented with respect to an algebra map $\epsilon_R : R \to k$.
We say that $(A, \epsilon_A)$ is an \emph{augmented $H$-cleft extension} over $R$,
if $A$ is $H$-cleft extension over $R$, and $\epsilon_A : A \to k$ is an algebra map extending $\epsilon_R$;
the inclusion $R \hookrightarrow A$ is thus an augmented algebra map. 
In this case we can and we do choose an invertible $H$-colinear map $\phi : H \to A$ so that
\begin{equation}\label{Ephi}
\mathrm{(i)}~\phi(1_H)=1_A,\quad \mathrm{(ii)}~\epsilon_A\circ \phi=\varepsilon_H.
\end{equation}
In this paper we call such a $\phi$ an \emph{$H$-section}; it is characterized as a unit-preserving $H$-colinear 
section of the $H$-colinear extension $\widetilde{\epsilon}_A =(\epsilon_A\otimes \mathrm{id}_H)\circ \rho_A : A \to H$ of $\epsilon_A$; see \eqref{Etilde}.
For (ii), one may replace a unit-preserving $\phi$ with 
$h\mapsto \epsilon_A(\phi^{-1}(h_{(1)}))\phi(h_{(2)})$,\ $H \to A$. 
One sees that $R \otimes_{\sigma}H$ then turns into an augmented $H$-cleft extension over $R$ with
respect to $\epsilon_R\otimes \varepsilon_H : R\otimes_{\sigma}H\to k$, which is indeed an algebra map.
It is isomorphic to $(A,\epsilon_A)$ through the isomorphism given by \eqref{ERHA}.
Here two augmented $H$-cleft extensions over $R$
are said to be \emph{isomorphic}, if there exists
an $H$-colinear, augmented $R$-algebra morphism (necessarily, isomorphism by \cite[Lemma 1.3]{M0}) between them.

\begin{lemma}\label{Laug}
Let $R^+:=\operatorname{Ker}(\epsilon_R : R \to k)$ denote the augmentation ideal, and suppose
that it is square-zero, $(R^+)^2=0$. Then
\begin{itemize}
\item the augmented $H$-cleft extensions over $R$, and
\item the Hochschild extensions of $H$ by $R^+\otimes H$ in the category $\mathsf{Vec}^H$ of $H$-comodules
\end{itemize}
are in one-to-one correspondence up to isomorphism and equivalence, where $R^+$ is regarded as a trivial $H$-comodule,
and $H$ thus coacts on $R^+\otimes H$ through the right tensor factor. 
\end{lemma}

\begin{proof}
To an augmented $H$-cleft extension $(A,\epsilon_A)$ over $R$, there 
corresponds the Hochschild extension
\begin{equation}\label{ER+H}
0 \to R^+ \otimes H \longrightarrow A \overset{\widetilde{\epsilon}_A}{\longrightarrow} H\to 0,
\end{equation}
of $H$ by $R^+\otimes H$ in $\mathsf{Vec}^H$,
where $\widetilde{\epsilon}_A$ denotes, as before, the $H$-colinear extension of $\epsilon_A$. 
To see that this is indeed such an extension, choose an $H$-section, and construct 
the $H$-crossed product $R \otimes_{\sigma} H$ as above.
Through the isomorphism given by \eqref{ERHA}, $\widetilde{\epsilon}_A : A \to H$
is identified with $\epsilon_R \otimes \mathrm{id}_H : R\otimes_{\sigma} H \to H$, whose kernel 
is $R^+\otimes H$, indeed; one also verifies that the $H$-section, which is identified with 
$h\mapsto 1\otimes h$, $H \to R\otimes_{\sigma}H$, is a unit-preserving 
$H$-colinear section of $\widetilde{\epsilon}_A$. 
In addition, if two augmented $H$-cleft extensions,
$(A,\epsilon_A)$ and $(A',\epsilon_{A'})$, over $R$ are isomorphic, 
we have isomorphic square-zero extensions, $(A,\widetilde{\epsilon}_A)$ and $(A',\widetilde{\epsilon}_{A'})$,
of $H$ in $\mathsf{Vec}^H$,
whence the corresponding Hochschild extensions are equivalent. 

Conversely, given a Hochschild extension
\begin{equation}\label{ER+Hi}
0\to R^+ \otimes H \overset{i}{\longrightarrow} A \overset{p}{\longrightarrow} H\to 0
\end{equation}
of $H$ by $R^+\otimes H$ in $\mathsf{Vec}^H$, one sees that a unit-preserving $H$-colinear 
section $j$ of $p$ is necessarily invertible (in $\operatorname{Hom}(H,A)$, since it is so modulo
the square-zero ideal $\operatorname{Hom}(H,R^+A)$).
Moreover, we have $A^{\operatorname{co}H}=R^+\oplus k=R$, and
$(A,\varepsilon_H\circ p)$ is an augmented $H$-cleft extension over $R$ with an $H$-section $j$. 
Clearly, mutually equivalent such Hochschild extensions give rise to mutually 
isomorphic augmented $H$-cleft extensions
over $R$. 

For $(A,\varepsilon_H\circ p)$ as above, the corresponding Hochschild extension is
the given one \eqref{ER+Hi}. 
Conversely, the Hochschild extension in the form \eqref{ER+H} gives rise to $(A, \epsilon_A)$. 
\end{proof}

In view of Proposition \ref{PH2}, it follows that the set of all isomorphism classes of 
the augmented $H$-cleft extensions over $R$ is in one-to-one correspondence with
$H_s^2(H, R^+)$. Since we are interested in the case where $R^+$ is one-dimensional, we
choose as $R$ the algebra 
\[ 
k[\tau]=k[T]/(T^2)
\]
of dual numbers, in which we have set $\tau:=T\, \operatorname{mod}(T^2)$. 
This is augmented with respect to the natural projection
$\epsilon_{k[\tau]} : k[\tau] \to k[\tau]/(\tau)=k$, and
the augmentation ideal $k[\tau]^+=k\tau$ is one-dimensional, indeed.

For this $k[\tau]$ we have the following. 

\begin{prop}[\text{see \cite[Proposition 1.9 (2)]{M3}}]\label{P1to1}
There is a natural one-to-one correspondence between
\begin{itemize}
\item 
the set of all isomorphism classes of 
the augmented $H$-cleft extension over $k[\tau]$, and
\item
$H_s^2(H,k)$. 
\end{itemize}
\end{prop}

Given an augmented $H$-cleft extension $(A,\epsilon_A)$ over $k[\tau]$, 
choose an $H$-section $\phi : H \to A$. 
We have the Hochschild extension given in \eqref{ER+H}, in which $\widetilde{\epsilon}_A$
has $\phi$ as a unit-preserving $H$-colinear section. 
In view of \eqref{Eumtilde} (see the description of the product $m_{\widetilde{E}}$ on $H \otimes H$),
we see that there uniquely arises a normalized symmetric Hochschild 2-cocycle $s: H \otimes H \to k$ such that
\begin{equation}\label{Ephiphi}
\phi(h)\phi(\ell)=\phi(h\ell)+s(h_{(1)},\ell_{(1)})\tau \, \phi(h_{(2)}\ell_{(2)}), \quad h, \ell \in H. 
\end{equation}
The cohomology class of $s$ in $H^2_s(H,k)$ corresponds to the isomorphism class of $(A,\epsilon_A)$.

\begin{rem}\label{Rnoncom}
For non-commutative $H$-comodule algebras, the notion 
of an \emph{augmented $H$-cleft extension} over $k[\tau]$
is defined without any change. The last cited result \cite[Proposition 1.9 (2)]{M1} tells, in fact, that 
the one-to-one correspondence given in Proposition \ref{P1to1} 
extends to a one-to-one correspondence
between 
\begin{itemize}
\item 
the set of all isomorphism classes of those non-commutative 
augmented $H$-cleft extension over $k[\tau]$ which are \emph{central}, i.e.,
including $k[\tau]$ in their centers, and
\item
the 2nd Hochschild cohomology $H^2(H,k)$ with coefficients in the trivial $H$-module $k$. 
\end{itemize}
\end{rem}

\subsection{Sample 1: $H$ is an example from \cite{T}}\label{subsec:sample1}

In this subsection we suppose $\operatorname{char}k=p>0$, and 
let $H$ be the Hopf algebra given in \cite[Example 1.7]{T}, denoted by $\widetilde{A}$. 
Thus, $H$ is, as an algebra, generated by an $\infty$-sequence of elements,
$y_1, y_2, \dots, y_i,\dots$, and is defined by the
relations
\[ 
y_1^p=y_2^{p^2}=\cdots =y_i^{p^i}=\cdots. 
\]
This is indeed a Hopf algebra, in which $y_i$ are all primitive, $\Delta_H(y_i)=y_i\otimes 1+1\otimes y_i$. 
It is shown in the cited example that $H_{\operatorname{red}}$ is a smooth Hopf algebra, and
the projection $H \to H_{\operatorname{red}}$ with nil (but not bounded nil) kernel does not 
split as an algebra map. 

One sees that $H$ has a linear basis, 
\[
y_1^{e_1}y_2^{e_2}\cdots y_i^{e_i}\cdots,
\]
where the exponents $e_i$ range so that 
\begin{equation}\label{Ebasis}
\text{(i)}\ e_1\ge 0,\ \, \text{(ii)}\ 0\le e_i<p^i\ \text{for}\ i\ge 2,\ \, \text{and}\ \, \text{(iii)}\ e_i=0\ \text{for}\ i\gg 0. 
\end{equation}

\begin{prop}\label{Psample1}
We have an isomorphism 
\begin{equation}\label{EH2s}
H_s^2(H, k)\simeq k^{\mathbb{N}}
\end{equation}
of vector spaces, 
where $k^{\mathbb{N}}=\prod_{i=0}^{\infty}k$ denotes the direct product of countably infinitely many 
copies of $k$. 
\end{prop}
\begin{proof}
Given an $\infty$-sequence $\boldsymbol{c}=(c_1,c_2,\dots,c_i,\dots)$ of scalars $c_i\, (\in k)$, let
$A_{\boldsymbol{c}}$ denote the $k[\tau]$-algebra which is generated by $z_1, z_2, \dots, z_i,\dots$,
and is defined by the relations
\[
z_i^{p^i}=z_{i+1}^{p^{i+1}}+c_i\tau,\quad i=1,2, \dots. 
\]
A simple use of Bergman's Diamond Lemma \cite[Section 10.3]{B} shows that $A_{\boldsymbol{c}}$ has a $k[\tau]$-free basis,
\[
z_1^{e_1}z_2^{e_2}\cdots z_i^{e_i}\cdots,
\]
where the exponents $e_i$ range just as in \eqref{Ebasis}. 
This $A_{\boldsymbol{c}}$ turns into an $H$-comodule algebra with respect to the coaction 
$\rho_{\boldsymbol{c}} : A_{\boldsymbol{c}}\to A_{\boldsymbol{c}}\otimes H$ defined by
\[
\rho_{\boldsymbol{c}}(z_i)=z_i\otimes 1 +1\otimes y_i,\quad i=1,2,\dots. 
\]
Equip $A_{\boldsymbol{c}}$ with the algebra map $\epsilon_{\boldsymbol{c}} : A_{\boldsymbol{c}}\to k$
defined by
\[
\epsilon_{\boldsymbol{c}}(\tau)=0\quad \text{and}\quad
\epsilon_{\boldsymbol{c}}(z_i)=0,\ \, i=1,2,\dots.
\]
Define a linear map $\phi_{\boldsymbol{c}} : H\to A_{\boldsymbol{c}}$ by
\begin{equation*}\label{Ephic}
\phi_{\boldsymbol{c}}(y_1^{e_1}y_2^{e_2}\cdots y_i^{e_i}\cdots)=z_1^{e_1}z_2^{e_2}\cdots z_i^{e_i}\cdots,
\end{equation*}
where the $e_i$ range so as in \eqref{Ebasis}. 
One sees that $\phi_{\boldsymbol{c}}$ is $H$-colinear,
and it satisfies the conditions (i)--(ii) in \eqref{Ephi}. In particular, it satisfies (i), and 
is, therefore, invertible by \cite[Lemma 9.2.3]{Sw}, since $H$
is irreducible, i.e., $1_H$ spans the coradical of $H$. 
Since the $k[\tau]$-linear extension 
\begin{equation}\label{Ektau}
k[\tau]\otimes H\to A_{\boldsymbol{c}},\quad x \otimes a \mapsto x\phi_{\boldsymbol{c}}(a). 
\end{equation}
of $\phi_{\boldsymbol{c}}$ is an $H$-colinear isomorphism, 
we have $A_{\boldsymbol{c}}^{\operatorname{co} H}=k[\tau]$. Thus, 
$A_{\boldsymbol{c}}=(A_{\boldsymbol{c}},\epsilon_{\boldsymbol{c}})$ is an augmented $H$-cleft extension over $k[\tau]$ with 
an $H$-section $\phi_{\boldsymbol{c}}$.

\begin{claim}
Choose two $\infty$-sequences $\boldsymbol{c}$ and $\boldsymbol{c}'$. 
If $A_{\boldsymbol{c}}$ and $A_{\boldsymbol{c}'}$ are isomorphic as augmented $H$-cleft extensions over $k[\tau]$, then $\boldsymbol{c}=\boldsymbol{c}'$.
\end{claim}

Indeed, suppose that $f : A_{\boldsymbol{c}'}\to A_{\boldsymbol{c}}$ is an isomorphism. This $f$ decomposes
uniquely into a composition
\[
A_{\boldsymbol{c}'}\overset{\simeq}{\longrightarrow}  k[\tau]\otimes H \overset{\simeq}{\longrightarrow} A_{\boldsymbol{c}},
\]
where the second isomorphism is \eqref{Ektau}. The first is necessarily in the form
\[
a \, (\in A_{\boldsymbol{c}'}) \mapsto (g \otimes \operatorname{id}_H)\circ \rho_{\boldsymbol{c}'}(a),
\]
where $g : A_{\boldsymbol{c}'}\to k[\tau]$ is a $k[\tau]$-linear map which is unit-preserving and augmented. 
So, we have $g(z'_i) =d_i\tau$ for some $d_i \in k$, and 
\[
f(z'_i)=z_i + d_i\tau,\quad i=1,2,\dots,
\]
where $z'_i$ denote the generators of $A_{\boldsymbol{c}'}$ analogous to the $z_i$ of $A_{\boldsymbol{c}}$. 
It follows that 
\[
f((z'_i)^{p^i} - (z'_{i+1})^{p^{i+1}})=(z_i + d_i\tau)^{p^i}-(z_{i+1}+d_{i+1}\tau)^{p^{i+1}},\quad i=1,2,\dots,
\]
which reduce to $c'_i\tau=c_i\tau$, or $\boldsymbol{c}'=\boldsymbol{c}$. This proves the claim.

\begin{claim}
Every augmented $H$-cleft extension over $k[\tau]$ is isomorphic to $A_{\boldsymbol{c}}$
for some $\boldsymbol{c}\in k^{\mathbb{N}}$.
\end{claim}

To prove this, let $(A,\epsilon_A)$ be an augmented $H$-cleft extension over $k[\tau]$. Choose
an $H$-section $\phi : H \to A$, and set 
\[
w_i:=\phi(y_i),\quad i=1,2,\dots. 
\]
Notice $\epsilon_A(w_i)=0$. By a simple computation one sees $w_i^{p^i}-w_{i+1}^{p^{i+1}}\in A^{\operatorname{co}H}=k[\tau]$, and that this last element is killed by $\epsilon_A$. It follows that
\[
w_i^{p^i}-w_{i+1}^{p^{i+1}}=c_i\tau, \quad i=1,2,\dots
\]
for some $c_i \in k$. Set $\boldsymbol{c}:=(c_1,c_2,\dots,c_i,\dots)\, (\in k^{\mathbb{N}})$. 
Then we see that
\[
h(z_i)=w_i,\quad i=1,2,\dots
\]
defines an $H$-colinear, augmented $k[\tau]$-algebra map $h : A_{\boldsymbol{c}}\to A$,
which is necessarily an isomorphism by 
\cite[Lemma 1.3]{M0}. This proves the claim. 

From the two claims above and Proposition \ref{P1to1} we obtain 
a one-to-one correspondence between the two sets in \eqref{EH2s}.
To $\boldsymbol{c}=(c_1,c_2,\dots,c_i,\dots)$ in $k^{\mathbb{N}}$, 
there corresponds the cohomology class of the normalized 2-cocycle 
$s_{\boldsymbol{c}} : H \otimes H \to k$ which arises from the $H$-section $\phi_{\boldsymbol{c}}$
through the relation \eqref{Ephiphi}. It satisfies
\begin{itemize}
\item[(i)]\ $s(y_i,y_i^r)=0$\ \, for\ \, $0\le r < p^i-1$, 
\item[(ii)]\ $s(y_i,y_i^{p^i-1})-s(y_{i+1},y_{i+1}^{p^{i+1}-1})=c_i$,
\end{itemize}
where $i=1,2,\dots$. Notice that
$s(y_i^r, y_i^t)=s(y_i, y_i^{r+t-1})$ for any positive integers $r$, $t$, since the 2-cocycles 
are naturally identified with the linear maps $S^2_H(H^+) \to k$; see the proof of Proposition
\ref{PH2C}. 
If a 2-cocycle $s$ has the properties (i)--(ii) above, then the corresponding 
augmented $H$-cleft extension over $k[\tau]$ is isomorphic to the $A_{\boldsymbol{c}}$, as is seen 
from the argument proving the last claim.
It follows that two 2-cocycles satisfying (i)--(ii) in common 
are cohomologous to each other. Therefore, we see
\[
s_{\boldsymbol{c}}+s_{\boldsymbol{c}'}\sim
s_{\boldsymbol{c}+\boldsymbol{c}'}\quad \text{and}\quad \lambda s_{\boldsymbol{c}}
\sim s_{\lambda \boldsymbol{c}},\ \, \lambda \in k, 
\]
where $\sim$ indicates ``cohomologous". This proves that the obtained one-to-one correspondence is
indeed a linear isomorphism. 
\end{proof}

\subsection{Sample 2: $H$ is a group algebra $kG$}\label{subsec:sample2}
We continue to suppose  $\operatorname{char}k=p>0$.
Recall from Proposition \ref{Pres} that
the restriction map $H_s^2(H,k) \to H_s^2(J,k)$ is surjective, and this follows from the result proved under the
assumption that $H$ is of finite type.
Here we wish to verify the result directly by explicit computations, when $H=kG$ is a finite-type commutative group algebra and $J$ is its Hopf (indeed, group) subalgebra. 

The symmetric 2nd Hochschild cohomology $H_s^2(kG,k)$ (resp., the 2nd Hochschild cohomology $H^2(kG,k)$)
is linearly isomorphic to the symmetric part of the 2nd group cohomology $H^2(G,k)_{\mathrm{sym}}$ 
(resp., the full 2nd group cohomology $H^2(G,k)$). 
Therefore, the computational results below may be standard, or even 
known, for experts in group theory. But our method of computations, which uses augmented cleft extensions
possessing linear structure and depends on the Diamond Lemma, would be useful. 

Let
$G$ be a finitely generated abelian group. 
Decompose $G$ into the direct sum $G=\bigoplus_{\alpha}G_\alpha$, where 
$\alpha$ runs through $\{\, \infty,\ \text{all pimes}\, \}$, 
of the torsion-free part $G_\infty$ and the $t$-primary parts $G_t$ for primes $t$. 
We have
\begin{equation}\label{Egrcoh}
H_s^2(kG,k)=\bigoplus_{\alpha}H_s^2(kG_\alpha, k),\ \,  \text{and}\ \, 
H_s^2(kG_\alpha, k)=0\ \text{if}\ \alpha\ne p,
\end{equation}
in view that $kG_{\alpha}$ is smooth unless $\alpha =p$. 
Therefore, we may and we do suppose that
$G$ is a non-trivial finite abelian $p$-group.
We present the operation on $G$ by addition. Let $o(g)$ denote the order
of $g \in G$, which is thus some power of $p$.

Choose a system $x_1,\dots,x_q$ of non-zero generators of $G$ such that 
\begin{equation}\label{Egl1}
G=\mathbb{Z}x_1\oplus\cdots \oplus \mathbb{Z}x_q.
\end{equation}
The number $q$ does not depend of choice of the systems.

\begin{prop}\label{Psample2}
We have a linear isomorphism
\begin{equation}\label{Eglisom}
H_s^2(kG, k)\simeq k^q.
\end{equation}
\end{prop}
\begin{proof}
Given an element $\boldsymbol{a}=(a_1,\dots,a_q)$ of $k^q$, define $\mathcal{A}_{\boldsymbol{a}}$ to
be the $k[\tau]$-algebra which is generated by $z_1,\dots,z_q$, and is defined by the
relation
\[
z_i^{o(x_i)}=1+a_i\tau,\quad 1\le i\le q.
\]
By the Diamond Lemma \cite[Section 10.3]{B}, this has a $k[\tau]$-free basis,
\[
z_1^{n_1}\cdots z_q^{n_q},\quad 1\le n_i <o(x_i),\ 1 \le i \le q.
\]
Equipped with the algebra map $\mathcal{A}_{\boldsymbol{a}}\to k$ determined by $\tau \mapsto 0$ and
$z_i \mapsto 1$, $1 \le i \le q$, $\mathcal{A}_{\boldsymbol{a}}$ turns into an augmented
$kG$-cleft extension over $k[\tau]$ with respect to the coaction given by 
\begin{equation}\label{Erhoa}
\rho_{\boldsymbol{a}} : \mathcal{A}_{\boldsymbol{a}}\to 
\mathcal{A}_{\boldsymbol{a}}\otimes kG,\quad \rho_{\boldsymbol{a}}(z_i)= z_i \otimes x_i,\ 1 \le i \le q.
\end{equation}
Indeed, a $kG$-section is given by the linear map 
$\phi_{\boldsymbol{a}} : kG \to \mathcal{A}_{\boldsymbol{a}}$ which sends each element
$x_1^{n_1}\cdots x_q^{n_q}$ of $G$, where $1\le n_i <o(x_i)$,\ $1 \le i \le q$, to the $k[\tau]$-free basis
element $z_1^{n_1}\cdots z_q^{n_q}$ of $\mathcal{A}_{\boldsymbol{a}}$. The associated (see \eqref{Ephiphi}) symmetric Hochschild 2-cocyle
$s : kG \otimes kG \to k$ is seen to have the properties
\begin{itemize}
\item[(i)]\ $s(x_i,x_i^t)=0$\ \, for\ \, $0\le t < o(x_i)-1$, 
\item[(ii)]\ $s(x_i,x_i^{o(x_i)-1})=a_i$,
\end{itemize}
where $1\le i \le q$. An argument analogous to the one of proving Proposition \ref{Psample1} shows
that every augmented cleft $kG$-cleft extension over $k[\tau]$ is isomorphic to 
some $\mathcal{A}_{\boldsymbol{a}}$,
where $\boldsymbol{a}\, (\in k^q)$ is uniquely determined, and the resulting one-to-one correspondence
$H_s^2(kG,k)\simeq k^q$ is a linear isomorphism. 
\end{proof}

Let $F$ be a non-trivial subgroup of $G$. 
Choose a system $y_1,\dots,y_r$ of non-zero generators of $F$ such that 
\begin{equation*}\label{Egl2}
F=\mathbb{Z}y_1\oplus\cdots \oplus \mathbb{Z}y_r,
\end{equation*}
similar to \eqref{Egl1}. 
Then one can present uniquely the elements $y_i$ so that
\[ 
y_i= \sum_{j=1}^q c_{ij}x_j,\quad \ 0\le c_{ij}<o(x_j),\ c_{ij}\in \mathbb{Z},
\]
where $1 \le i \le r$. Note that
\[
o(y_i)=\max\{ \, o(c_{ij}x_j)\mid 1 \le j\le q \, \}.
\]

\begin{lemma}
The two systems of the generators can be re-chosen so that
\begin{itemize}
\item[(i)] $c_{ij}=0$ if $i > j$;
\item[(ii)] $o(y_i)=o(c_{ii}x_i)$,\ $1\le i\le r$;
\item[(iii)] $o(y_1)\ge \cdots\ge o(y_r)$.
\end{itemize}
\end{lemma}

This is easily proved by induction on $q$. As for $x_1,\dots, x_q$, it suffices to re-number them
without re-choosing the elements.

\begin{rem}
By (i) we necessarily have $r \le q$.
By (ii), for each $1\le i\le q$, $o(y_i)$ divides $o(x_i)$, and $c_{ii}$ is a multiple of $o(x_i)/o(y_i)$
by some non-zero integer, say, $m_i$ that is prime to $p$. 
We can and we do re-choose $y_1,\dots,y_r$ so that $m_i=1$, or namely,
\begin{itemize}
\item[(iv)] $c_{ii} = o(x_i)/o(y_i),\quad 1\le i \le r$.
\end{itemize}
We have only to replace each $y_i$ with $\ell_i y_i$, where $\ell_i$ is an integer
such that $\ell_i m_i\equiv 1$ mod $o(x_i)$. 
\end{rem}

\begin{prop}\label{Pgrres}
The restriction map 
\[
\mathrm{res} : H_s^2(kG, k)\to H_s^2(kF, k)
\]
is surjective.
\end{prop}
\begin{proof}
Recall from the proof of Proposition \ref{Psample2} that every augmented $kG$-cleft extension over $k[\tau]$ is in the form
$\mathcal{A}_{\boldsymbol{a}}$, where $\boldsymbol{a}\in k^q$. 
It has the $kG$-section $\phi_{\boldsymbol{a}} : kG \to \mathcal{A}_{\boldsymbol{a}}$.
Define an augmented subalgebra of 
$\mathcal{A}_{\boldsymbol{a}}$ by
\[
\mathcal{B}:=\rho_{\boldsymbol{a}}^{-1}(\mathcal{A}_{\boldsymbol{a}}\otimes kF),
\]
where $\rho_{\boldsymbol{a}}$ denotes the coaction \eqref{Erhoa} on $\mathcal{A}_{\boldsymbol{a}}$.
As a general fact, this is a $kF$-comodule algebra with respect to the coaction 
$\rho_{\boldsymbol{a}}|_{\mathcal{B}}$. Moreover, $\mathcal{B}$ is the augmented $kF$-cleft
extension over $k[\tau]$, with a $kF$-section $\phi_{\boldsymbol{a}}|_{kF} : kF \to \mathcal{B}$. 
Clearly, the assignment $\mathcal{A}_{\boldsymbol{a}}\mapsto \mathcal{B}$ 
is identified with the restriction of their associated 
(Hopf, and also Hochschild) 2-cocycles. 

For $1\le i\le r$, let 
\[
w_i:=\phi_{\boldsymbol{a}}(y_i)=z_i^{c_{ii}}z_{i+1}^{c_{i,i+1}}\cdots z_q^{c_{iq}}\, (\in \mathcal{B}). 
\]
We have an isomorphism $H_s^2(kF,k)\simeq k^r$ which is analogous 
to \eqref{Eglisom}, and is now determined by the system $y_1,\dots,y_r$ of generators.
It sends the $\mathcal{B}$, naturally identified with a cohomology class, to 
the element $\boldsymbol{b}=(b_1,\dots,b_r)$ of $k^r$ determined by 
\begin{equation}\label{Ew1}
w_i^{o(y_i)}=1+b_i\tau,\quad 
1\le i\le r. 
\end{equation}
To present $\boldsymbol{b}$ explicitly, define an $r\times q$ matrix $T=\begin{pmatrix} t_{ij}\end{pmatrix}$ by
\[
t_{ij}=\begin{cases} c_{ij}\frac{o(y_i)}{o(x_j)}, &\text{if}\ i\le j;\\
0, &\text{otherwise}. \end{cases}
\]
Notice that $T$ is an integral matrix such that 
the $r \times r$ submatrix $T'$ of $T$ consisting of the first $r$ columns 
is an upper triangular with diagonal entries $1$ by (iv), whence it has an integral inverse matrix.
We see 
\[
(z_j^{c_{ij}})^{o(y_i)}=1+t_{ij}a_j\tau, \quad i\le j, 
\]
whence
\begin{equation}\label{Ew2}
w_i^{o(y_i)}= 1+\sum_{j=1}^q t_{ij}a_j\tau,\quad 1\le i\le r.
\end{equation}
As an additional remark, in case $i<j$, the integer $t_{ij}$ is divided by $p$ if and only if $o(y_i)>o(c_{ij}x_j)$,
whence the terms $t_{ij}a_j\tau$ in $j$ such that $o(y_i)>o(c_{ij}x_j)$ are zero.
Now, from \eqref{Ew1}--\eqref{Ew2} we see ${}^t\boldsymbol{b}
=T\, {}^t\! \boldsymbol{a}$. It follows that the restriction map in question is identified,
through the relevant isomorphisms, with the linear map $k^q \to k^r$ (between vector spaces of column vectors) 
given by the matrix $T$ with entries specialized to $k$. 
Since the matrix has full rank by the invertibility of $T'$, 
the last linear map and thus the restriction map are surjective.
\end{proof}

Notice from \eqref{Egrcoh} that the conclusion
of Proposition \ref{Pgrres} holds, 
more generally, when $G$ is a finitely generated abelian group, and $F$ is its subgroup. 

Recall Remark \ref{Rnoncom}, and notice the restriction map extends to that 
\[
\operatorname{res} : H^2(kG, k) \to H^2(kF, k)
\]
of (non-symmetric) 2nd Hochschild cohomologies. 
To conclude this section we show that this extended restriction map can be non-surjective.

\begin{prop}
The extended restriction map above is not surjective under the situation where
$G=\mathbb{Z}x_1\oplus \mathbb{Z}x_2$ is a finite group of order $p^3$ and $F=\mathbb{Z}y_1\oplus \mathbb{Z}y_2$ is its subgroup of order $p^2$ such that
\[
o(x_1)=p^2,\ \, o(x_2)=p,\ \, y_1=px_1,\ \, y_2=x_2.
\]
\end{prop}
\begin{proof}
Choose arbitrarily an element $b \in k$, and let $B_b$ be the non-commutative $k[\tau]$-algebra 
which is generated by $w_1$, $w_2$, and 
is defined by the relations
\[
w_1^p=w_2^p=1,\quad w_2w_1=(1+b\tau)w_1w_2.
\]
Then we see by using the Diamond Lemma that $B_b$ has a $k[\tau]$-free basis, $w_1^{n_1}w_2^{n_2}$, where $0\le n_i<p,\ i=1,2$,
and is indeed a non-commutative augmented, central $kF$-cleft extension over $k[\tau]$ (see Remark \ref{Rnoncom}) with respect to 
the algebra map $B_b\to k$,\ $\tau \mapsto 0$, $w_i \mapsto 1$ and the coaction $B_b\to B_b \otimes kF$,\ 
$w_i\mapsto w_i\otimes y_i$, where $i=1,2$.  Suppose that this $B_b$ arises by restriction 
from a non-commutative augmented, central $kG$-cleft extension, say, $A$ over $k[\tau]$. 
Choose a $kG$-section $\phi : kG \to A$, and let $z_1:=\phi(x_1)$,\
$z_2:=\phi(x_2)$. Then we see that
$A$ includes $B_b$ so that
\[
z_1^p= (1+c_1\tau)w_1,\quad z_2=(1+c_2\tau)w_2
\]
for some $c_1, c_2 \in k$. 
But, since $z_2z_1=(1+c_3\tau)z_1z_2$ for some $c_3\in k$, it follows that
\[
\begin{split}
(1+(c_1+c_2)\tau)w_2w_1&=z_2z_1^p=(1+c_3\tau)^pz_1^pz_2\\
&=z_1^pz_2=(1+(c_1+c_2)\tau)w_1w_2,
\end{split}
\]
whence we must have $b=0$.  Therefore, if $b \ne 0$, then $B_b$ does not arise from any non-commutative augmented, 
central $kG$-cleft extension over $k[\tau]$; this shows the desired non-surjectivity.   
\end{proof}

\section{Proof of Theorem \ref{T3rd}}\label{Sec5}

This section is divided into two subsections, according to the two parts of the theorem.

\subsection{Proof of Part 1, characteristic zero case}\label{subsec5:charzero}

We work in $\mathsf{sVec}$. We do not assume $\operatorname{char}k=0$ until explicitly mentioned. 
Recall that every object $X$ of $\mathsf{sVec}$ decomposes uniquely into the direct sum $X=X_0\oplus X_1$
of the even component $X_0$ and the odd component $X_1$. The objects $X$ such that $X=X_0$ 
form a symmetric monoidal full subcategory of $\mathsf{sVec}$ isomorphic to $\mathsf{sVec}$.

Let $H$ be a Hopf
algebra in $\mathsf{sVec}$, and let $\mathfrak{G}=\operatorname{Spec}H$ denote the corresponding affine group scheme.
There are naturally associated 
\begin{equation}\label{EHbar}
\overline{H}:=H/H_1H\quad \text{and}\quad W^H:=(T^*_{e}(\mathfrak{G}))_1.
\end{equation}
Here, $\overline{H}=H/H_1H$ denotes the quotient of $H$ divided by the
Hopf ideal generated by the odd component $H_1$ of $H$, which is characterized as
the largest ordinary quotient Hopf algebra of $H$. On the other hand, 
$W^H=(T^*_{e}(\mathfrak{G}))_1$ denotes the odd component 
of the cotangent super-vector space $T^*_{e}(\mathfrak{G})=H^+/(H^+)^2$
of $\mathfrak{G}$ at the identity $e$. 
By Theorem 4.5 of \cite{M2} (proved when $\operatorname{char}k\ne 2$)
we have an algebra isomorphism
\begin{equation}\label{Etenpr}
H\overset{\simeq}{\longrightarrow}\wedge(W^H) \otimes \overline{H}
\end{equation}
in $\mathsf{sVec}$, 
which can be chosen so as to be counit-preserving and $\overline{H}$-colinear, in addition; 
refer to \cite[Remark 5.9]{M4} for a more acceptable proof. The cited proofs prove, in fact, the opposite-sided version of the isomorphism above.
In characteristic $\operatorname{char} k=2$ (recall that then, the commutativity additionally requires every odd element to be
square-zero), the isomorphism follows from \cite[Theorem 5.7]{MS}, if one assumes that $H$ is of finite type.
It can be generalized, removing the assumption, by following the above cited proofs.

Let $R$ be an algebra in $\mathsf{sVec}$ with an ideal $I$. 
Since every odd element of $I$ is square-zero, it follows that $I$ is nil if and only if the even component
$I_0$ of $I$ is a nil ideal of the ordinary algebra $R_0$. 

In view of \eqref{Etenpr}, it suffices for our purpose to prove that an (ordinary) algebra map 
$\overline{H}\to (R/I)_0=R_0/I_0$ can
lift to some algebra map $\overline{H}\to R_0$, since an algebra morphism $\wedge(W^H)\to R/I$,
being restricted to a linear map $W^H\to (R/I)_1$ that can lift to some linear map
$W^H \to R_1$, can lift to the algebra morphism $\wedge(W^H) \to R$ which uniquely extends the last linear map.
A conclusion is that we may (and we do) work in $\mathsf{Vec}$. 

Now, assume $\operatorname{char}k=0$.
Given an (ordinary) algebra $R$ with a nil ideal $I$, an (ordinary) Hopf algebra $H$ and an algebra map $f:H \to R/I$,
we have the same pull-back diagram as \eqref{Epb}, in which 
$\varpi: T \to H$ is now an algebra surjection with nil kernel. 
It suffices to prove that this $\varpi$ has a section, whose composition with
$\omega$ is seen to be a desired lift. 

Consider a pair $(J, s)$ of a Hopf subalgebra $J$ of $H$, and a \emph{section} $s : J \to \varpi^{-1}(J)$,
by which we mean a section of the algebra surjection $\varpi|_{\varpi^{-1}(J)}: \varpi^{-1}(J)\to J$.
An example is the trivial $(k, \mathrm{id}_k)$. Introduce a natural order to the (non-empty) set of all those pairs,
by defining $(J_1,s_1)\le (J_2,s_2)$ if $J_1\subset J_2$ and $s_1=s_2|_{J_1}$. By Zorn's Lemma we have
a maximal pair, say, $(J_{\max}, s_{\max})$. It remains to prove $J_{\max}=H$.
On the contrary we assume $J_{\max}\subsetneqq H$, and wish to show that there exists a pair properly greater than $(J_{\max},s_{\max})$; the contradiction implies the desired equality. 
One can choose a finite-type Hopf subalgebra $K$ of $H$ which is not included in $J_{\max}$.  
Define $J:=J_{\max}\cap K$. Then, $J$ is a Hopf subalgebra of $K$, which is of finite type by 
\cite[Corollary 3.11]{Tak}. 
One point of the present proof is: we have a canonical isomorphism 
\begin{equation}\label{EKJ}
J_{\max}\otimes_J K\simeq J_{\max}K,\quad x \otimes_K y \mapsto xy 
\end{equation}
by \cite[Proposition 6]{T0}. Notice that $s_{\max}|_J : J \to \varpi^{-1}(J)$ is a section. 

We claim that 
this $s_{\max}|_J$ extends to some section $K \to \varpi^{-1}(K)$. This can be restated
so that $s_{\max}|_J$ extends to some $J$-algebra section
$K \to \varpi^{-1}(K)$, with $\varpi^{-1}(K)$ regarded as a $J$-algebra through $s_{\max}|_J$.
Since $K$ is of finite type as a $J$-algebra, there exists a finite-type $J$-subalgebra, say, $A$ of $\varpi^{-1}(K)$
such that $\varpi|_A : A \to K$ maps onto $K$. Let $Q:=K/J^+K$ be the quotient Hopf algebra 
of $K$ which corresponds to $J$. We have
a $K$-algebra isomorphism similar to \eqref{EHJQ},
\begin{equation*}\label{EKJQ}
K\otimes_J K \overset{\simeq}{\longrightarrow} K \otimes Q,\quad
x \otimes_K y \mapsto xy_{(1)}\otimes y_{(2)}\, \operatorname{mod}(J^+K),
\end{equation*}
where the source is regarded as a $K$-algebra through 
$K =K\otimes_J J \hookrightarrow K\otimes_J K$. 
Another point of the proof is: the Hopf algebra $Q$ is smooth
since $\operatorname{char} k=0$; see Theorem \ref{T1st} (1). 
By the base change property of smoothness, $K\otimes_JK\, (\simeq K\otimes Q)$ is smooth over $K$,
i.e., smooth in the category of $K$-algebras.
It follows by \cite[Proposition 5.1]{MOT} that $K$ is smooth (indeed, $Q$-smooth) over $J$. 
Notice that the kernel of $\varpi|_A : A \to K$ is finitely generated, whence it is a nilpotent ideal. 
Then we see that there exists a desired $J$-algebra section, say, $t : K \to \varpi^{-1}(K)$
which maps into $A$. 

In view of \eqref{EKJ}, we see that $s_{\max}$ and $t$ amount to a section
\[
J_{\max}K \to \varpi^{-1}(J_{\max}K), \quad xy\mapsto s_{\max}(x)t(y),
\]
which together with $J_{\max}K$ form a pair properly greater than $(J_{\max},s_{\max})$.
This completes the proof of Part 1 of the theorem.

\subsection{Proof of Part 2, positive characteristic case}\label{subsec5:charp}

Assume $\operatorname{char}k=p>0$. Let $\mathscr{C}$ be $\mathsf{sVec}$ or 
$\mathsf{Ver}_p^{\mathrm{ind}}$. Recall from Section \ref{S1.4} the additional assumption 
posed when $\mathscr{C}=\mathsf{Ver}_p^{\mathrm{ind}}$ that $p\ge 5$, and $k$ is algebraically closed. 

Suppose $\mathscr{C}=\mathsf{Ver}_p^{\mathrm{ind}}$. Let us recall from 
Ostrik \cite{O} and Venkatesh \cite{Ven} some basics on it. 
This is a semisimple abelian category which contains precisely $p-1$ simple objects
including the unit object $\mathbf{1}$. (Compare with the fact that 
$\mathsf{sVec}$ contains precisely two simples, i.e., the even and the odd 1-dimensional 
super-vector spaces.)
By \cite[Lemma 2.20]{Ven} the symmetric $p$-th
power $S^p(L)$ of every simple object $L$ other than $\mathbf{1}$ equals zero. Given an object $X$ of $\mathsf{Ver}_p^{\mathrm{ind}}$, 
let $X_0$ (resp., $X_{\ne 0}$) denote the sub-object of $X$ which is the sum of all simple
sub-objects that are (resp., are not) 
isomorphic to $\mathbf{1}$. Then we have $X=X_0\oplus X_{\ne 0}$. 
All the objects $X$ such that $X=X_0$ (resp., $X$ includes only simples that are isomorphic to either
$\mathbf{1}$ or another specific one, the $L_{p-1}$ in \cite{O} or \tc{\cite{Ven})}
form a symmetric monoidal full subcategory of $\mathsf{Ver}_p^{\mathrm{ind}}$ which is isomorphic to
$\mathsf{Vec}$ (resp., $\mathsf{sVec}$). 

Given an algebra $R$ in $\mathsf{Ver}_p^{\mathrm{ind}}$ with an ideal $I$, $R_0$ is an ordinary subalgebra of $R$, and $I_0$ is an ideal of $R_0$. We see that
$I$ is bounded nil if and only if $I_0$ is so. This is true for algebras and their ideals in $\mathsf{sVec}$, as well. 

Let us be given a Hopf algebra $H$ in $\mathsf{Ver}_p^{\mathrm{ind}}$, and the corresponding affine group scheme
$\mathfrak{G}=\operatorname{Spec}H$. The analogous objects to those in \eqref{EHbar} are
\begin{equation}\label{EHbar1}
\overline{H}:=H/H_{\ne 0}H\quad \text{and}\quad W^H:=(T^*_{e}(\mathfrak{G}))_{\ne 0}. 
\end{equation}
The former $\overline{H}$ is an ordinary Hopf algebra.
Theorem 5.1 of \cite{M4} proves that there is a (counit-preserving $\overline{H}$-colinear) 
algebra isomorphism 
\begin{equation}\label{Etenpr1}
H\overset{\simeq}{\longrightarrow}S(W^H) \otimes \overline{H},
\end{equation}
which is analogous to \eqref{EHbar1}.
Here, $S(W^H)$ denotes the symmetric algebra on $W^H$; it is an analogue 
of the exterior algebra $\wedge(W^H)$ in $\operatorname{sVec}$, which is, indeed, the symmetric algebra in 
the category. We remark that the cited theorem proves, in fact, the opposite-sided version of the isomorphism above,
replacing $S(W^H)$ with another isomorphic algebra $\Gamma(W^H)$. 

Corollary 1.6 of \cite{T} proves that if $H$ is a Hopf algebra 
in $\mathsf{Vec}$ for which Condition (f) is satisfied, or namely, the Frobenius map is injective,
then it has the stronger smoothness as claimed.  Since such an $H$ is smooth by Theorem \ref{T1st} (2),
the desired result holds for ordinary Hopf algebras. 
This, together with \eqref{Etenpr}, \eqref{Etenpr1}, proves the desired result in general, 
as is seen from the argument of the fourth paragraph of the last subsection.

\appendix

\section{A simple proof of triviality of torsors over an algebraically closed field}\label{appendixA}

In this appendix we work in $\mathsf{Vec}$, except in Remark \ref{RTan}.

Wibmer \cite{Wib} gave an elementary proof of the fact that a torsor under an affine group scheme over an algebraically
closed field is necessarily trivial; this fact is the same as the uniqueness of the so-called Tannakian groups over a field such as above. We give a (hopefully) simpler proof of the triviality of torsors,
which uses the same idea as was used to prove Theorem \ref{T3rd} (1); see Section \ref{subsec5:charzero}. 

Suppose that $k$ is an arbitrary field, and $\mathfrak{G}=\operatorname{Spec}H$ is an affine group scheme
represented by a Hopf algebra $H$. 
Every torsor under $\mathfrak{G}$ is necessarily affine, and it is uniquely represented 
by an $H$-comodule algebra $A$, which is a \emph{twisted 
form} of $H$, i.e., $A$ becomes isomorphic to $H$ after some base field extension. Such an $A$ is precisely
what is called an $H$-\emph{Galois extension} over $k$; it is by definition a non-zero $H$-comodule
algebra $A=(A, \rho_A)$ such that the $A$-linear extension
\begin{equation}\label{Ebeta}
A \otimes A \to A \otimes H,\quad a\otimes b\mapsto a\rho_A(b)
\end{equation}
of the structure map $\rho_A :A\to A\otimes H$ is an isomorphism (of $H$-comodule algebras); 
the isomorphism, specialized through
an algebra map $A \to L$ to some extension field $L$ of $k$,  
shows that $A$ is a twisted form of $H$, and conversely, for such a twisted form,
the map above is isomorphic since it turns so after some base field extension.  
Therefore, the above-mentioned triviality is translated as follows, and in fact, we
are going to give a simple proof of this. 

\begin{prop}\label{Ptor}
Suppose that $k$ is an algebraically closed field, and let $H$ be a Hopf algebra. Then an $H$-Galois extension 
over $k$ is necessarily trivial, i.e., it is isomorphic to $H$ as an $H$-comodule algebra. 
\end{prop}
\begin{proof}
Suppose that $A=(A,\rho_A)$ is an $H$-Galois extension over $k$. 
From \eqref{Ebeta} notice
\begin{equation}\label{EkAcoH}
k=A^{\operatorname{co}H}\, (:=\{ \, x \in A \mid \rho_A(x)=x\otimes 1\, \}). 
\end{equation}
We may suppose that $H$ is not of finite type, since $A$ is  
trivial by Hilbert's Nullstellensatz, in case $H$ is of finite type; see \cite[Section 18.3]{Wa}. 
Given a Hopf subalgebra $J$ of $H$, define
\begin{equation}\label{EAJ}
A_J:=\rho_A^{-1}(A \otimes J). 
\end{equation}
This turns naturally into a $J$-comodule algebra, and is, in fact, a $J$-Galois extension over $k$.
For the isomorphism \eqref{Ebeta} restricts to an isomorphism $A \otimes A_J \overset{\simeq}{\longrightarrow} A\otimes J$
of $J$-comodule algebras, which shows by specialization that $A_J$ is a twisted form of $J$. 
In particular, we have
$A_J^{\operatorname{co}J}=k$; see \eqref{EkAcoH}. 
Suppose that we are given an $H$-colinear algebra map $f : J \to A$.
Then we see that $f$ maps into $A_J$, and there results a $J$-colinear algebra map $J \to A_J$, 
through which $A_J$ is a $J$-module, and is, moreover, a $J$-Hopf module \cite[Section 4.1]{Sw}. 
It follows by the Hopf-module Theorem \cite[Theorem 4.1.1]{Sw} that the resulting is, in fact, an isomorphism  
\begin{equation}\label{EJAJ}
J=A_J^{\operatorname{co}J}\otimes J \overset{\simeq}{\longrightarrow} A_J
\end{equation}
of $J$-comodule algebras; see the proof of Proposition \ref{Plci}.

Analogously to the proof of Theorem \ref{T3rd} (1), let us consider a pair $(J, f)$ of a Hopf subalgebra $J$ of $H$ and an $H$-colinear
algebra map $f : J \to A$. A trivial example is $(k, k \hookrightarrow A)$.
The set of all those pairs, into which we introduce 
the natural order defined by $(J_1, f_1)\le (J_2,f_2)$ if $J_1\subset J_2$ and 
$f_1=f_2|_{J_1}$, has a maximal pair, $(J_{\max}, f_{\max})$, by Zorn's Lemma. 
As is seen from the cited proof, it remains to prove the following.

\begin{claim}
Given finite-type Hopf subalgebras $J$ and $K$ of $H$ such that $J\subset K$, every $H$-colinear
algebra map $J \to A$ can extend to such a map $K \to A$. 
\end{claim}

Indeed, in order to see $J_{\max}=H$, which will complete the proof, we suppose on the contrary that 
$J_{\max} \subsetneqq H$, 
and choose arbitrarily a finite-type Hopf subalgebra $K$ of $H$ which is not included in $J_{\max}$. The
claim applied to $J_{\max}\cap K$ and $K$ gives an $H$-colinear algebra map 
from $J_{\max} K$\, ($=J_{\max}\otimes_{J_{\max}\cap K}K$, see \eqref{EKJ}) to $H$ which extends
$f_{\max}$; the extended map, paired with $J_{\max} K$, gives a pair properly greater than $(J_{\max}, f_{\max})$. 
This contradiction shows the desired equality.

To prove the claim let us be given an $H$-colinear algebra map $f : J \to A$, which results in an 
isomorphism $J \overset{\simeq}{\longrightarrow} A_J$ of $J$-comodule algebras such as in \eqref{EJAJ}. 
As was seen in the first paragraph of the proof, 
one can choose an isomorphism $K \overset{\simeq}{\longrightarrow} A_{K}$ of $K$-comodule algebras, which we
identify with the identity map $\operatorname{id}_{K}$ of $K$. Then $f$ is identified with an automorphism
$f : J \overset{\simeq}{\longrightarrow} J$\, ($=\Delta_K^{-1}(K\otimes J)$, see \eqref{EAJ}) of 
the $J$-comodule algebra $J$, which is necessarily in the form
\[
f : J \to J,\quad f(x)=g(x_{(1)})x_{(2)}, 
\]
where $g=\varepsilon_J\circ f$, or $f=\widetilde{g}$ in the notation of \eqref{Etilde}. 
It remains to prove that this can extend to some automorphism 
$f': K \overset{\simeq}{\longrightarrow} K$ of the $K$-comodule algebra $K$. Since such an $f'$ is necessarily in 
the form $f'(x)=g'(x_{(1)})x_{(2)}$ with $g': K \to k$ an algebra map, we need to prove that
the algebra map $g : J\to k$ can extend to an algebra map $K \to k$. Indeed, this 
is possible by Hilbert's Nullstellensatz since $K$ is faithfully flat over $J$ by \cite[Theorem 3.1]{Tak}. 
\end{proof}

\begin{rem}\label{RTan}
Suppose $\mathscr{C}=\mathsf{sVec}$ or $\mathsf{Ver}_p^{\mathrm{ind}}$; see \eqref{Ecat}. 
The result just proven holds more generally in $\mathscr{C}$.
To see this, suppose that $H$ is a Hopf algebra in $\mathscr{C}$, and let $\mathfrak{G}=\operatorname{Spec}H$
be the corresponding affine group scheme. 
Since every non-zero algebra in $\mathscr{C}$ has a field as a non-zero quotient 
ordinary algebra, it follows that in $\mathscr{C}$, an $H$-Galois extension over the unit object
$\mathbf{1}$ is the same as
a twisted form of the $H$-comodule algebra $H$. 
Those forms are classified by the 1st Amitsur cohomology
$H^1(\mathbf{1},\mathfrak{G})_{\mathscr{C}}$ constructed in $\mathscr{C}$. But this is naturally identified with
the ordinary $H^1(k,\mathfrak{G}_0)$, or namely, we have
\[
H^1(\mathbf{1},\mathfrak{G})_{\mathscr{C}}=H^1(k,\mathfrak{G}_0),
\]
where $\mathfrak{G}_0=\operatorname{Spec}\overline{H}$ is the affine group scheme 
corresponding to the ordinary Hopf algebra $\overline{H}$ (see \eqref{EHbar}, \eqref{EHbar1}) associated with $H$. For the complex 
for computing $H^1(\mathbf{1},\mathfrak{G})_{\mathscr{C}}$, which is in the form
\[
\begin{xy}
(0,0)   *++{\mathfrak{G}(L)\rightrightarrows \mathfrak{G}(L\otimes L)\rightrightarrows
\mathfrak{G}(L\otimes L\otimes L),}  ="1",
(3.1,-2.1)  *++{\to} ="2"
\end{xy}
\]
is naturally identified with the one 
\[
\begin{xy}
*++{\mathfrak{G}_0(L)\rightrightarrows \mathfrak{G}_0(L\otimes L)\rightrightarrows
\mathfrak{G}_0(L\otimes L\otimes L)},
(4.4,-2.1)  *++{\to} 
\end{xy}
\]
for computing $H^1(k,\mathfrak{G}_0)$, where $L$ ranges through all extension-fields of $k$; see 
\cite[Section 5.2]{MOT}. 
Therefore, the proved Proposition \ref{Ptor} implies the generalized result in the categories. 
\end{rem}

\vspace{3mm}

\noindent
\textbf{Funding}\quad
The second-named author is supported by Japan Society for Promotion of Science, KAKENHI, Grant Numbers 23K03027.

\vspace{3mm}

\noindent
\textbf{Data Availability}\quad All data generated or analyzed during this study are included in this article.

\vspace{5mm}

\noindent
{\large{\textbf{Declarations}}}

\vspace{3mm}

\noindent
\textbf{Conflicts of Interest}\quad The authors have no conflict of interest to declare that are relevant to this article.


\begin{thebibliography}{99}

\bibitem{AMS}
A.~Ardizzoni,\ C.~Menini,\ D.~Stefan,\
\emph{Hochschild cohomology of algebras in monoidal
categories and splitting morphisms of bialgebras},\ 
Trans. Amer. Math. Soc. {\bf 359} (2007),\ 991--1044. 

\bibitem{A}
L.~L.~Avramov,\
\emph{Locally complete intersection homomorphisms and a conjecture of Quillen on the vanishing of cotangent homology},\ 
Ann. of Math. {\bf 150} (1999), no.\! 2, 455--487.

\bibitem{B}
G.~M.~Bergman,\
\emph{The diamond lemma for ring theory},\ 
Adv. in Math. {\bf 29} (1978),\ 178--218.

\bibitem{BCM}
R.~J.~Blattner,\ M.~Cohen,\ S.~Montgomery,\ 
\emph{Crossed products and inner actions
of Hopf algebras},\ Trans. Amer. Math. Soc. {\bf 298} (1986), 671–680.




\bibitem{CEO}
K.~Coulembier,\ P.~Etingof,\ V.~Ostrik,
\emph{On Frobenius exact symmetric tensor categories}, With an appendix by A. Kleshchev,
Ann. of Math. (2) \textbf{197} (2023), no. 3, 1235--1279.

\bibitem{D}
P.~Deligne,\ \emph{Cat\'{e}gories tensorielles},\ Mosc. Math. J. {\bf 2} (2002), no.\! 2, 227–248.

\bibitem{DG} 
M.~Demazure,\ P.~Gabriel,\ 
\emph{Groupes alg\'{e}briques,\ tome I},\ Masson~$\And$~Cie,\ Paris; North-Holland, Amsterdam, 1970.

\bibitem{Doi}
Y.~Doi,\ 
\emph{Structure of relative Hopf modules},\ Comm. Algebra {\bf 11} (1983), no.\! 3, 243--255. 

\bibitem{DT} 
Y.~Doi,\ M. Takeuchi,\
\emph{Cleft comodule algebras for a bialgebra},\
Comm. Algebra {\bf 14} (1986),\ no.\! 5,\ 801–817.



\bibitem{MR}
J.~Majadas,\ A.~G.~Rodicio,\
\emph{Smoothness, regularity and complete intersection},\ London Math. Soc. Lecture Note Series, vol.~373,\
Cambridge Univ. Press, 2010. 

\bibitem{M0}
A.~Masuoka,\
\emph{Cleft extensions for a Hopf algebra generated by a
 nearly primitive element},\
Comm. Algebra {\bf 22} (1994), no.\! 11, 4537--4559. 


\bibitem{M1}
A.~Masuoka,\ 
\emph{Hopf cohomology vanishing via approximation by Hochschild cohomology},\
Noncommutative Geometry and Quantum Groups (Warsaw, 2001), 111--123, Banach Center Publ. vol.~61, Polish Acad. Sci., Warsaw, 2003.

\bibitem{M2}
A.~Masuoka,\
\emph{The fundamental correspondences in super affine groups and super formal groups},\
J.\ Pure\ Appl.\ Algebra {\bf 202} (2005), 284--312.


\bibitem{M3}
A.~Masuoka,\
\emph{Abelian and non-abelian second cohomologies of quantized enveloping algebras},\ J. Algebra {\bf 320} (2008), 1--47.

\bibitem{M4}
A.~Masuoka,\
\emph{Harish-Chandra pairs and affine algebraic group schemes in the Verlinde category, revisited},\
submitted; manuscript available on request to the author.


\bibitem{MOT}
A.~Masuoka,\ T.~Oe,\ Y.~Takahashi,\ 
\emph{Torsors in super-symmetry}, J. Noncom. Geometry {\bf 19}(2025), no.\! 3,\ 931–969.


\bibitem{MO} 
A.~Masuoka,\ T.~Oka,\ 
\emph{Unipotent algebraic affine supergroups and nilpotent Lie superalgebras},\
Algebr. Represent. Theory \textbf{8} (2005), no.\! 3, 397--413. 


\bibitem{MS} 
A.~Masuoka,\ T.~Shibata,\ 
\emph{On functor points of affine supergroups},\
J.\ Algebra {\bf 503} (2018) 534--572.

\bibitem{MW}
A.~Masuoka,\ D.~Wigner,\
\emph{Faithful flatness of Hopf algebras},\ 
J. Algebra {\bf 170} (1994), 156--164.

\bibitem{Mon}
S.~Montgomery,\  \emph{Hopf Algebras and Their Actions on Rings},\
CBMS Regional Conference Series in
Mathematics, vol.~82, Amer. Math. Soc., Providence, RI, 1993.

\bibitem{O}
V.~Ostrik,\
\emph{On symmetric fusion categories in positive characteristic},\ Selecta Math. (N.S.) {\bf 26}
(2020), no.\! 3, Paper No. 36, 19 pp.

\bibitem{Sw}
M.~E.~Sweedler,\
\emph{Hopf algebras},\
Math. Lecture Note Ser.,\ W. A. Benjamin, Inc., New York, 1969.

\bibitem{Tak}
M.~Takeuchi,\
\emph{A correspondence between Hopf ideals and sub-Hopf algebras},\ 
Manuscripta Math. \textbf{7} (1972), 251–270.


\bibitem{Tak1} M.~Takeuchi,\ 
\emph{Formal schemes over fields},\ 
Comm.\ Algebra \textbf{5} (1979), no.\! 14,\ 1483--1528.


\bibitem{T0}
M.~Takeuchi,\
\emph{Relative Hopf modules---equivalences and freeness criteria},\
J.\ Algebra \textbf{60} (1979),\ 452--471.

\bibitem{T}
M.~Takeuchi,\
\emph{Commutative Hopf algebras and cocommutative Hopf algebras in positive characteristic},\
J.\ Algebra \textbf{79} (1982),\ 375--392.

\bibitem{Stack}
The Stack Project Authors,\ \emph{Stack Project},\ 
available at: https://stacks.math.columbia.edu/

\bibitem{Ven}
S.~Venkatesh,\
\emph{Harish-Chandra pairs in the Verlinde category in positive
characteristic},\ 
Internat. Math. Res. Not. IMRN (2023), no. \, 18, 15475--15536. 

\bibitem{Wa}
W.~C.~Waterhouse,\
\emph{Introduction to affine group schemes},
Graduate Texts in Mathematics, vol.~66, 
Springer-Verlag, New York, 1979. 

\bibitem{Wei}
C. A.~Weibel,\ \emph{An introduction to homological algebra}, paperback edition, Cambridge
Studies in Advanced Mathematics vol.~38, Cambridge Univ. Press, 1995.

\bibitem{Wib}
M.~Wibmer,
\emph{A remark on torsors under affine group schemes},\
Transform. Groups \textbf{30} (2025), 447--454.

\end{thebibliography}
\end{document}